\documentclass[reqno,centertags, 11pt, draft]{article}

\usepackage{amsmath,amsthm,amscd,amssymb}

\usepackage{latexsym}

\usepackage{mathrsfs}

\newcommand{\bbN}{{\mathbb{N}}}

\newcommand{\bbR}{{\mathbb{R}}}

\newcommand{\bbE}{{\mathbb{E}}}

\newcommand{\bbL}{{\mathbb{L}}}

\newcommand{\bbC}{{\mathbb{C}}}

\newcommand{\dmu}{{\rm d} \mu}

\newcommand{\no}{\nonumber}

\newcommand{\wti}{\widetilde  }

\newcommand{\supp}{\text{\rm{supp}}}

\newcommand{\eps}{\varepsilon}

\newcommand{\beq}{\begin{equation}}

\newcommand{\eeq}{\end{equation}}

\newcommand{\ba}{\begin{align}}

\newcommand{\ea}{\end{align}}

\DeclareMathOperator{\Tr}{Tr}

\DeclareMathOperator*{\EE}{{\mathbb{E}}}

\DeclareMathOperator*{\Var}{Var}

\allowdisplaybreaks

\numberwithin{equation}{section}

\newtheorem{theorem}{Theorem}[section]

\newtheorem{proposition}[theorem]{Proposition}

\newtheorem{lemma}[theorem]{Lemma}

\newtheorem{corollary}[theorem]{Corollary}

\theoremstyle{definition}

\newtheorem{definition}[theorem]{Definition}

\theoremstyle{remark}

\newtheorem{remark}{Remark}[section]

\date{}
\begin{document}
\sloppy
\title{The Nevai condition and a local law of large numbers for orthogonal polynomial ensembles}
\author{Jonathan Breuer\footnote{ Einstein Institute of Mathematics, The Hebrew University of Jerusalem,
Jerusalem 91904, Israel. E-mail: jbreuer@math.huji.ac.il. Supported in part by the US-Israel Binational Science Foundation (BSF) Grant no.\ 2010348 and by the Israel Science Foundation (ISF) Grant no.\ 1105/10.} \and Maurice Duits\footnote{Department of Mathematics, Royal Institute of Technology (KTH), Lindstedtsv\"agen 25, SE-10044 Stockholm, Sweden. Supported by the grant KAW 2010.0063 from the Knut and Alice Wallenberg Foundation.}}
\maketitle
\begin{abstract}
We consider asymptotics of orthogonal polynomial ensembles, in the macroscopic and mesoscopic scales. We prove both global and local laws of large numbers (analogous to the recently proven local semicircle law for Wigner matrices) under fairly weak conditions on the underlying measure $\mu$. Our main tools are a general concentration inequality for determinantal point processes with a kernel that is a self-adjoint projection, and a strengthening of the Nevai condition from the theory of orthogonal polynomials. 
\end{abstract}

\section{Introduction}

Let $\mu$ be a probability measure on $\bbR$ with finite moments. The Orthogonal Polynomial Ensemble (=OPE) of size $n\in \bbN$ associated with $\mu$ is the probability measure on $\bbR^n$ given by 
\beq \label{eq:defOPE}
\frac{1}{Z_n} \prod_{i>j} (\lambda_i-\lambda_j)^2 {\rm d} \mu(\lambda_1) \cdots {\rm d}\mu(\lambda_n),
\eeq
where $Z_n$ is a normalizing constant. In recent years many models from probability theory and combinatorics have been shown to give rise to OPE's. These include non-colliding random walks, growth models, last passage percolation and eigenvalues of certain invariant random matrix ensembles. Perhaps the most famous example is that of the Gaussian Unitary Ensemble, where $\dmu_n(x)=\sqrt{\frac{n}{2\pi}}e^{-\frac{n}{2}x^2}{\rm d}x$ (here $\mu$ is $n$-dependent). For a review of various important models leading to OPE's see, e.g., \cite{K}.

A considerable amount of attention has been devoted to understanding the asymptotics of OPE's (as $n \rightarrow \infty$) in the microscopic scale, namely inside an interval whose size is of the order of the mean distance between points.
%(for $\mu$ independent of $n$, this is typically $\sim \frac{1}{n}$ in the interior of $\supp(\mu))$.
 In particular, the phenomenon of \emph{universality}, where the microscopic correlations are independent of specific properties of $\mu_n$, has been shown to hold for increasingly large classes of measures. In this context, around points in the `bulk',  this means that the correlation kernel converges to the sine kernel. We do not attempt a review of this topic here and merely cite a (small) portion of relevant works \cite{DKMVZ1,DKMVZ2, findley, LeLu, lubinsky1, lubinskyUniv2, Pas97, Pas08, Simon-ext, totik}. In these works universality was proven for points in the interior of the support of the absolutely continuous part of $\mu$. Recently, universality was shown to occur also for some measures supported on Cantor sets \cite{als} and even for certain singular continuous measures \cite{Bjat}. Universality holds in other contexts as well. Of relevance here is the recently proven universality for the eigenvalue correlations for Wigner random matrices (see \cite{EY} for a review).

This paper focuses on asymptotic properties of the (random) empirical measure $n^{-1} \sum \delta_{x_j}$ for OPE's on the macro- and mesoscopic scales. Roughly speaking, on the macroscopic scale we consider intervals that contain a number of points that is of order $\sim n$, while on the microscopic scale  this number is finite. We refer to the intermediate scales as the mesoscopic scales.   

On the macroscopic scale the mean empirical measure has been extensively studied. Under fairly weak assumptions it has a weak limit that is characterized in terms of an equilibrium problem. We will review some results that are relevant to us in Section 2.2. For OPE's coming from unitary ensembles in random matrix theory,  large deviation principles have been derived \cite{BA}, \cite[Prop. 2.6.1]{AGZ} from which it follows, in particular, that the empirical measure has an almost sure weak limit. In the same context, a Central Limit Theorem for the fluctuations around the limit was established in \cite{Jduke}.
In \cite{hardy} almost sure convergence of the moments for the empirical measure  was recently proved under a certain growth condition on the Jacobi coefficients associated to $\mu=\mu_n$. 

Compared to the the macro- and microscopic scale, the mesoscopic scale for OPE  is relatively unexplored.  Nevertheless, some important studies have been carried out before in the context of various other models of random matrix theory.  For instance, Central Limit Theorems for linear statistics on these  scales have been derived for the  GUE \cite{BdMK}, Wigner matrices \cite{BdMK2}, the classical compact groups \cite{SoshlocalCLT} and in the context of Dyson's Brownian Motion \cite{DJ}. Another notable example is  the recent work mentioned above, regarding universality for Wigner random matrices. In \cite{esy} convergence to the limiting empirical distribution has been shown to hold down to almost microscopic scales (`local semicircle law'). This result was an important step towards the above-mentioned proof of universality, but is of considerable independent interest as well. We refer to such a result as a `local law of large numbers'. 

 Our main results in this paper are concentration inequalities for OPE's and laws of large numbers on all scales.  There is a vast amount of literature on concentration inequalities and we mention  \cite{AGZ} as  a general reference in the context of random matrix theory.   In order for us to give an example of the type of theorems we shall prove, we define the \emph{linear statistic} associated with a function $f: \bbR \rightarrow \bbR$, by 
\beq \label{eq:deflinstat}
X_{f}^{(n)} =\sum f(x_j),
\eeq 
where $(x_1, x_2, \ldots, x_n)$ is random from the OPE. The determinantal structure of the OPE (see Section \ref{DetPP} below) implies
\beq \no
\EE X_f^{(n)} =\int f(x) K_n(x,x) {\rm d}\mu(x)
\eeq
where
\beq
K_n(x,y)=\sum_{j=0}^{n-1} p_{j}(x)p_j(y),
\eeq
(with $(p_j)_{j=0}^\infty$ the orthogonal polynomials corresponding to $\mu$), is the Christoffel-Darboux kernel (see \cite{Simon-CD} for a review of its properties and some applications). The following theorem is part of Theorem \ref{thm:generalmacroLLN} that we will prove later. In fact, the first part  is a general statement that is true for any determinantal point process with a kernel that is a self-adjoint projection of finite rank. The general statement is given in Theorem \ref{th:generalexponentialboundswithvariance}.

\begin{theorem}[Global Law of Large numbers for OPE] \label{thm:LLN}
There exists a universal constant, $A>0$, such that for any measure with finite moments, $\mu$, any bounded function, $f$, and any $\varepsilon>0$
\beq \label{eq:expboundmacro}
P\left( \left|\frac{1}{n} X_f^{(n)}-\frac{1}{n}\EE X_f^{(n)} \right|\geq \eps \right)\leq 2 \exp \left(- n \eps \min\left( \frac{ \eps}{ 8A \|f\|_\infty^2}, \frac{1}{6 \|f \|_\infty} \right) \right)
\eeq
for all $n \in \bbN$. 

In particular, if $\mu_n$ is a sequence of such measures for which also $\frac{K_n(x,x)}{n}{\rm d} \mu_n(x)$ has a weak limit, $\nu$, and $f$ is bounded and continuous, then
 \beq \label{eq:LLN}
\lim_{n\to \infty} \frac{1}{n} X_f^{(n)}= \int f(x) \ {\rm d} \nu(x),
\eeq 
almost surely. 
\end{theorem} 

\begin{remark}\label{rem:constant}
We emphasize that the concentration inequality \eqref{eq:expboundmacro} is uniform in the measure $\mu$. The constant $A>0$ that we derive in the proof is $A=2{\rm e}^2\sum_{m=0}^\infty ({\rm e}/3)^m(m+2)^{3/2}$, which is the result of a rather rough estimate.
\end{remark}
\begin{remark}
For continuous $f$ \eqref{eq:expboundmacro} can be slightly simplified. Moreover, for Lipschitz functions $f$ the result can be substantially improved in the sense that we have an exponential bound for $X_f^{(n)}-\EE X_f^{(n)}$ without any normalization, see  \eqref{eq:expboundmacrocont} in Theorem \ref{thm:generalmacroLLN}. This is due to the strong repulsion between the points coming from the Vandermonde determinant in \eqref{eq:defOPE}. 
\end{remark}

\begin{remark}
The convergence \eqref{eq:LLN} is a law of large numbers. The class of fixed, compactly supported measures for which $\frac{K_n(x,x)}{n}{\rm d} \mu(x)$ has a weak limit is quite large. It includes all \emph{regular} measures (we discuss these further below) and measures arising as spectral measures of ergodic Jacobi matrices. In case this limit does not exist, we still have a law of large numbers along subsequences where there is a limit (the existence of such subsequences is guaranteed by compactness). In any case, these limits are the same as the limits of the normalized counting measure of the zeros of the orthogonal polynomials associated with $\mu$ \cite{simonDuke}. Thus, they are always continuous measures. In some cases, a law of large numbers can in fact be proved for bounded non-continuous functions (see Remark \ref{rem:convergence} below). 
\end{remark}

We want to point out that  concentration inequalities are typically derived under stronger assumptions (see for example \cite[Prop. 4.4.26]{AGZ} for OPE's from unitary ensembles where a strong convexity assumption is needed) on the underlying probability measure and Lipschitz functions $f$. The concentration inequalities  that we derive  in this paper, like \eqref{eq:expboundmacro} (and \eqref{eq:expboundmacrocont} for Lipschitz functions $f$), hold  under fairly weak assumptions on $\mu$. In fact, for \eqref{eq:expboundmacro} we only need that $\mu$ has finite moments.  The reason for this is the determinantal structure of the OPE. Indeed, the conclusions of Theorem \ref{thm:LLN} in fact hold in a context that is more general than that of OPE. The main ingredient in the proof is a general concentration inequality that holds for any determinantal point process with a kernel that is a finite rank self-adjoint projection (see Theorem \ref{th:generalexponentialbounds} below). This inequality essentially allows us to bound all moments of a linear statistic by its variance. This variance, for any point process with such a kernel, can then be bounded in a simple way, using the finite rank of the kernel leading to a proof of \eqref{eq:expboundmacro}. 

But we can do more. This same bound also proves a local analog for some of the mesoscopic scales. When a limit exists in an appropriate sense, this then implies a (local) law of large numbers by standard methods. To properly formulate a local version of Theorem \ref{thm:LLN} we define, for $\alpha>0$ and $x^* \in \bbR$,
\beq \label{eq:linstat}
X_{f,\alpha,x^*}^{(n)}=\sum_{j=1}^n f(n^\alpha(\lambda_j-x^*)).
\eeq
This local linear statistic probes the point process around the point $x^*$ at the scale $n^{-\alpha}$ with a function $f$. Now note that the dependence of \eqref{eq:expboundmacro} on $f$ enters only through $\|f\|_\infty$. Thus, by defining $\wti{f}(x)=f\left(n^\alpha(x-x^*)\right)$, we may apply it to the scaled linear statistic (since $\| \wti{f}\|_\infty=\|f\|_\infty$). Replacing $\eps$ by $\eps N/n$ we immediately find the following result. %Note that for $\alpha=0$  we retrieve the global scale. 

\begin{theorem} \label{thm:semiLLLN}
There exists a universal constant, $A>0$, such that for any measure with finite moments, $\mu$, any $x^*\in \bbR$, any $\alpha>0$,  any bounded function, $f$,  any $\varepsilon>0$, any $N\in \bbR$ and any $n\in \bbN$,
\beq \label{eq:boundhalf}
P\left(\frac{1}{N} \left| X_{f,\alpha,x^*}^{(n)}-\EE X_{f,\alpha,x^*}^{(n)} \right|\geq \eps \right)\leq 2 \exp \left(- \varepsilon \min \left( \frac{ \eps N^2}{ 8A n \|f\|_\infty^2 }, \frac{N}{6\|f\|_\infty} \right) \right)
\eeq
\end{theorem}

To draw meaningful conclusions from \eqref{eq:boundhalf}, we need to choose  the normalizing constant $N$  in an appropriate way. For a local law of large numbers we choose $N=N_{n,\alpha,x^*}$ where $N_{n,\alpha,x^*}$ is of the same order as the expected number of points in  $x^* + n^{-\alpha} \supp f$ (see the discussion in Remark \ref{rem:convergence} regarding convergence of the scaled mean).  Note that Theorem \ref{thm:semiLLLN} is meaningful only as long as $N_{n,\alpha,x^*} /\sqrt n \to \infty$ as $n\to \infty$.   Hence we cannot get to the  microscopic scale, but only down to intervals with more than $\sqrt{n}$ points.

In many cases,  $n^{-1} K_n(x,x) {\rm d}\mu(x)$ has a weak limit that is absolutely continuous with respect to  Lebesgue measure in an interval around $x^*$. Moreover,  if the density on that interval is positive and bounded (i.e.\ $x^*$ is in `the bulk'), then the proper normalization is $N= n^{1-\alpha}$. In these cases, Theorem \ref{thm:semiLLLN}  is only meaningful for $0<\alpha< \frac{1}{2}$.  However, in the bulk we expect a similar result to hold for $0\leq \alpha < 1$ (in analogy with the local semi-circle for Wigner matrices mentioned above).  In order to extend the analysis to smaller scales ($\alpha \geq 1/2$) more assumptions are needed. Here we use the structure of the OPE and, in particular, certain properties of orthogonal polynomials.  The analysis results in exponential bounds that are stronger than \eqref{eq:boundhalf}.

 For the sake of simplicity of presentation, let us assume for now that $\mu$ has compact support. Write ${\rm d}\mu(x)=w(x){\rm d}x+\dmu_{\textrm{sing}}(x)$, with $\mu_{\textrm{sing}}$ the part of $\mu$ that is singular with respect to Lebesgue measure. For $p_j$, the $j$'th orthogonal polynomial w.r.t.\ $\mu$, write $p_j(x)=\gamma_j x^j+\ldots$ with $\gamma_j>0$. We say that $\mu$ is \emph{regular} (in the sense of Stahl-Totik-Ullmann, \cite{ullman, st, simonCap}) if 
\beq \label{eq:regular}
\lim_{n \rightarrow \infty} \gamma_n^{1/n}=\frac{1}{\textrm{cap}\left( \supp (\mu) \right)}
\eeq
where $\textrm{cap}(E)$ denotes the potential theoretic capacity of a set $E \subseteq \bbR$. As discussed in Section \ref{Reg} below, regular measures have various nice properties. In particular, if $\mu$ is regular, then 
\beq \label{eq:regularconv}
\frac{K_n(x,x) \dmu(x)}{n} \rightarrow {\rm d}\nu_{\textrm{eq};\mu}(x)
\eeq
weakly, where $\nu_{\textrm{eq}; \mu}$ is the equilibrium measure of $\supp(\mu)$. 

\begin{theorem}[Local Law of Large Numbers for OPE]\label{thm:LLLN}
Let $\dmu(x)=w(x){\rm d}x+\dmu_{\textrm{sing}}$ be a regular measure with compact support, $E \subseteq \bbR$. Suppose $I$ is a closed interval in the interior of $E$ such that $\mu$ is absolutely continuous on a neighborhood of $I$ and $w$ is continuous and nonvanishing on $I$. 
Then for any compactly supported, bounded function, $f$, with a finite number of points of discontinuity, and any $\eps>0$ 
\beq \label{eq:LCI2}
\mathbb P\left( n^{\alpha}  \left|\frac{X_{f,\alpha,x^*}^{(n)}}{n} 
-\frac{\EE X_{f,\alpha,x^*}^{(n)} }{n}\right|\geq \eps\right)\leq 2
{\rm e}^{-\frac{\eps  n^{(1-\alpha)}}{6 \|f\|_\infty} }, 
\eeq
for  $n\in \bbN$ sufficiently large. Hence in particular,
\beq \label{eq:asconvergence2}
\lim_{n \rightarrow \infty} n^\alpha \left (\frac{1}{n} X_{f,\alpha,x^*}^{(n)}-\int_I f\left(n^\alpha (x-x^*)\right){\rm d}\nu_{\textrm{eq},\mu} (x) \right)=0,
\eeq
almost surely. 
\end{theorem}

\begin{remark}
An analogous result for varying measures (not necessarily compactly supported) is formulated and proven in Section \ref{alphanevai} (see Theorem \ref{LLLN1} and Corollary \ref{cor:varyingLLLN}).
\end{remark}

Theorem \ref{thm:LLLN}, together with its analogs for varying measures, relies on a property we call `the $\alpha$-Nevai condition', which we show implies a bound on the variance of the linear statistic at scale $\alpha$. We then use this bound together with the concentration inequality discussed above to prove a local law of large numbers at all mesoscopic scales for any measure satisfying the $\alpha$-Nevai condition. 

The $\alpha$-Nevai condition is a strengthening of the Nevai condition introduced by Nevai in \cite{nevai} in order to study ratio asymptotics for Christoffel functions. This condition says that $\frac{K_n(x,y)^2}{K_n(x,x)}\dmu(y)$ converges to a delta measure at $x$. It has been extensively studied within the theory of orthogonal polynomials \cite{BLS, CMN, dellavecchia, LO, lubinskyNevai, LN, Nevai2, NTZ, Szwarc, Szwarc2, Zhang} and has been shown to hold uniformly in the support of $\mu$, for a rather large class of measures $\mu$ (for example, $\mu$ belonging to the Nevai class of a finite gap set \cite{LN, NTZ, Zhang, Szwarc2, BLS}). Moreover, it has been shown to hold in measure on the absolutely continuous part of any compactly supported measure \cite{lubinskyNevai}. However, there are examples of measures (even regular ones) where the Nevai condition fails for $x$ in a large subset of $\supp(\mu)$ (see, e.g., \cite{BLS, Szwarc}).  Conjecture 1.4 of \cite{BLS} states that for any compactly supported $\mu$, the Nevai condition holds for $\mu$-a.e.\ $x$.

The $\alpha$ version that we introduce here says that the convergence to a delta measure is `fast' and uniform in a certain sense (determined by $\alpha$). We shall prove that if the OPE associated with $\mu$ has universal microscopic limits uniformly on a neighborhood of $x^*$ then $\mu$ obeys an $\alpha$-Nevai condition at $x^*$ (for any $\alpha<1$). 
As an offshoot of our analysis, we deduce that the Nevai condition follows from pointwise universality, wherever $\mu$ is absolutely continuous. This is discussed further in Section \ref{alphanevai}.

Returning to the local law of large numbers for OPE, an informal summary of this discussion is that uniform microscopic universality, i.e.\ convergence to the sine kernel, (together with a certain local regularity of $\mu$) implies a local law of large numbers for all scales $\alpha<1$. 
The only property of the sine kernel, however, used to deduce the $\alpha$-Nevai condition is the fact that its integral is $1$. Thus, we believe is should be possible to extend our analysis here also to OPE's where universality is not known. 

Before concluding this section, we would like to remark that while we use bounds on the variance in this paper to deduce concentration inequalities and (local) laws of large numbers for linear statistics, such bounds could be useful in other contexts as well. One example of such a use is the derivation of Central Limit Theorems. Such theorems have been found, in the context of Lipschitz functions, in other related models in random matrix theory \cite{BdMK,BdMK2,SoshCLT}. Remarkably, in these known examples, the convergence to normal variables occurs without a normalizing factor (in all scales). This result is expected to be universal. In particular, one expects the variance for linear statistics of Lipschitz functions to be bounded in all scales under very general conditions. Proposition \ref{prop:lipschitz} is an encouraging first step in this direction.
 
The rest of this paper is structured as follows. Section \ref{preliminaries} contains a brief overview of what we need from the theory of determinantal point processes and from the theory of orthogonal polynomials. Section \ref{concentration} contains the statement and proof of the general concentration inequality underlying our analysis in both the macroscopic and mesoscopic scale. Section \ref{sec:variance} contains some general bounds on the variance of both scaled and non-scaled (macroscopic) linear statistics and the consequences of these bounds. In particular, Theorems \ref{thm:LLN} and \ref{thm:semiLLLN} are proved in Section 4. Finally, Section \ref{alphanevai} has our treatment of the $\alpha$-Nevai condition and its connection to bounds on the variance. Theorem \ref{thm:LLLN} is proved in Section 5.

In addition to Theorems \ref{thm:LLN}, \ref{thm:semiLLLN}, and \ref{thm:LLLN} presented in this Introduction, we would also like to single out Theorem \ref{th:generalexponentialbounds} which has our general concentration inequality, Theorem \ref{th:generalexponentialboundswithvariance} dealing with a generalization of Theorem \ref{thm:LLN} in the context of determinantal processes with a kernel that is a finite rank self-adjoint projection, Proposition \ref{prop:lipschitz} regarding the variance of $X_f^{(n)}$ for Lipschitz functions, Theorem \ref{thm:weaknevai} dealing with the Nevai condition; and Corollary \ref{cor:varyingLLLN} regarding the local law of large numbers for varying measures. We should also point out Theorems \ref{thm:generalmacroLLN} and  \ref{LLLN1} of which Theorems \ref{thm:LLN} and \ref{thm:LLLN} are, respectively, special cases. 

\subsection*{Acknowledgments}
We thank Kurt Johansson and Yoram Last for useful discussions.
%%%%%%%%%%%%%%%%%%%%%%%%%%%%%%%%% SECTION 2 %%%%%%%%%%%%%%%

\section{Preliminaries} \label{preliminaries}

%%%%%%%%%%%%%%%%%%%%%%%%%%%%%%%%% SECTION 2.1 %%%%%%%%%%%%%%

\subsection{Determinantal point processes} \label{DetPP}

As in the Introduction, we let $\mu$ be a probability measure with finite moments. We denote the orthogonal polynomials of $\mu$ by $\{p_j\}_{j=0}^\infty$. Namely, $p_j(x)=\gamma_j x^j+\ldots$ is a polynomial of degree $j$ with positive leading coefficient, $\gamma_j$, and for all $i,j$ 
\beq \no
\int p_j(x) p_i(x) \dmu(x)=\delta_{i,j}.
\eeq
Recall the definition of the associated OPE in \eqref{eq:defOPE}.  It is well-known that the OPE is an example of a determinantal point process. In such processes the $k$-point correlation functions     (or marginal densities)  are  given by determinants of $k\times k$ matrices determined by a function of two variables, called the correlation kernel. For more information and background material  we refer to \cite{BorDet,HKPV,J4,K,L,Sosh,Sosh2}.   As we will briefly illustrate below, it turns out that the OPE  is a determinantal point process with correlation kernel
\beq \no
K_n(x,y)=\sum_{j=0}^{n-1} p_j(x)p_j(y),
\eeq
which is  the Christoffel-Darboux (CD) kernel associated to the orthogonal polynomials. Note that the CD kernel is the kernel of the projection in $L^2(\dmu)$ onto the subspace spanned by $\{1,x,\ldots,x^{n-1}\}$. In particular, this implies the \emph{reproducing property}:
\beq \label{eq:RepCD}
K_n(x,y)=\int K_n(x,z)K_n(z,y)\dmu(z).
\eeq
Another formula that we shall have occasion to use is the Christoffel-Darboux formula
\beq \label{eq:CDForm}
K_n(x,y)=\frac{\gamma_{n-1}}{\gamma_n}\frac{p_n(x)p_{n-1}(y)-p_n(y)p_{n-1}(x)}{x-y}.
\eeq
In some of the bounds for the variance that we prove below, it is important to control the factor $\frac{\gamma_{n-1}}{\gamma_n}$. We note here that if $\mu$ has compact support then this factor is bounded.

  By noting that $\prod_{i>j} (\lambda_i-\lambda_j)^2$ is the square of a Vandermonde determinant, the probability measure \eqref{eq:defOPE} can be written as  a product of two determinants and after some linear algebra   we can rewrite \eqref{eq:defOPE} as 
\beq \label{eq:OPECD}
\begin{split}
&\frac{1}{Z_n} \prod_{i>j} (\lambda_i-\lambda_j)^2 {\rm d} \mu(\lambda_1) \cdots {\rm d}\mu(\lambda_n) \\
&\quad=\frac{1}{n!} \det\left(K_n(\lambda_i,\lambda_j) \right)_{1 \leq i,j \leq n} {\rm d} \mu(\lambda_1) \cdots {\rm d}\mu(\lambda_n)
\end{split}
\eeq
Moreover, by the reproducing property and \eqref{eq:OPECD} one can verify that  the marginal densities (or, up to a constant, the $k$-point correlations) are given by 
\beq 
\begin{split}
&\frac{1}{Z_n} \underbrace{\int\cdots \int}_{n-k \text{ times}} \prod_{i>j} (\lambda_i-\lambda_j)^2 {\rm d} \mu(\lambda_{k+1}) \cdots {\rm d}\mu(\lambda_n) \\
&\quad=\frac{(n-k)!}{n!} \det\left(K_n(\lambda_i,\lambda_j) \right)_{1 \leq i,j \leq k}.
\end{split}
\eeq
From here it also follows that   for any bounded $\phi$ we have 
\begin{multline}\label{eq:fredholm1}
\bbE \left[ \prod_{j=1}^n (1+\phi(x_j)) \right]\\
= \sum_{k=0}^n \frac{1}{k!}\int \cdots \int \phi(x_1)\cdots \phi(x_k)\det \left(K_k(x_i,x_j)\right)_{i,j=1}^k {\rm d}\mu(x_1) \cdots {\rm d}\mu(x_k).
\end{multline}
Note that the right-hand side equals the Fredholm determinant for the operator with kernel $\phi(x)K_n(x,y)$ on $\bbL _2(\mu)$ and therefore we can write
\beq\label{eq:fredholm2}
\bbE \left[ \prod_{j=1}^n (1+\phi(x_j)) \right]=\det (1+\phi K_n). 
\eeq
By taking $\phi={\rm e}^{tf}-1$ we see that the right-hand sides of \eqref{eq:fredholm1} and \eqref{eq:fredholm2} are the moment generating functions (or Laplace transforms) for the linear statistic $X_f^{(n)}$ associated with a function $f$ (recall \eqref{eq:deflinstat}). From here it follows in particular that 
\begin{align}\label{eq:trace}
\EE X_f^{(n)} &=\int f(x) K_n(x,x) {\rm d}\mu(x)\\
\Var X_f^{(n)}&= \int f(x)^2 K_n(x,x) {\rm d}\mu(x)\no\\
&\qquad-\iint f(x) f(y) K_n(x,y) K_n(y,x) {\rm d}\mu(x) {\rm d}\mu(y).
\end{align}
These identities hold for general determinantal point processes. Moreover,  since for the OPE we have the additional $K^*_n=K_n$ and $K_n^2=K_n$ we can write the variance also as
\beq \label{eq:variance2}
\Var X_f=\frac12 \iint \left(f(x)-f(y)\right)^2 K_n(x,y)^2 {\rm d}\mu(x) {\rm d}\mu(y),
\eeq
which turns out to be useful for our purposes.

Similarly (recall \eqref{eq:linstat}),
\beq \no
\EE X_{f,\alpha,x^*}^{(n)} =\int f\left(n^\alpha(x-x^*)\right) K_n(x,x) {\rm d}\mu(x),\\
\eeq
and
\beq \label{eq:AlphaVariance}
\begin{split}
&\Var X_{f,\alpha,x^*}^{(n)}= \int f\left(n^\alpha(x-x^*)\right)^2 K_n(x,x) {\rm d}\mu(x) \\
&\quad -\iint f\left(n^\alpha(x-x^*)\right) f\left(n^\alpha(y-x^*)\right) K_n(x,y) K_n(y,x) {\rm d}\mu(x) {\rm d}\mu(y) \\
&=\frac12 \iint \left(f\left(n^\alpha(x-x^*)\right)-f\left(n^\alpha(y-x^*)\right)\right)^2 K_n(x,y)^2 {\rm d}\mu(x) {\rm d}\mu(y).
\end{split}
\eeq
In Section \ref{concentration} we will use the representations \eqref{eq:fredholm2} and \eqref{eq:variance2} to obtain a general concentration inequality for determinantal point processes corresponding to kernels that are self-adjoint projections.

%%%%%%%%%%%%%%%%%%%%%%%%%%%%% SECTION 2.2 %%%%%%%%%%%%%%%%%% 

\subsection{Convergence of the mean density and universality for OPE's} \label{Reg}

In everything that follows, when given a measure $\mu$ we write the decomposition of $\mu$ into absolutely continuous and singular parts with respect to Lebesgue measure as $\dmu(x)=w(x) {\rm d}x+\dmu_{\textrm{sing}}$.

The limiting behavior of OPE's as $n\to \infty$ has been well-studied under various assumptions on $\mu$. When $\mu$ is a fixed, compactly supported measure, the mean empirical distribution, $\EE \left( \frac{1}{n} \sum_{j=1}^n \delta_{\lambda_j}(x) \right)= \frac{1}{n}K_n(x,x)\dmu(x)$, has subsequential weak limits by compactness. As noted in the Introduction, two classes of measures for which convergence is known are the spectral measures of ergodic Jacobi matrices on the one hand \cite{als}, and regular measures on the other. Recall we say that $\mu$ is \emph{regular} if 
\beq 
\lim_{n \rightarrow \infty} \gamma_n^{1/n}=\frac{1}{\textrm{cap}\left( \supp (\mu) \right)}
\eeq
where $\textrm{cap}(E)$ denotes the potential theoretic capacity of a set $E \subseteq \bbR$. For $\mu$ with $ \supp(\mu) \subseteq [-1,1]$, this class was singled out by Ullman \cite{ullman}. The case of general $\mu$ was studied by Stahl and Totik \cite{st}. For a review of the theory of regular measures on the line and the unit circle see \cite{simonCap}. 

One particularly useful property of regular measures is \eqref{eq:regularconv}, namely, the weak limit of the mean empirical distribution is the potential theoretic equilibrium measure of $\supp(\mu)$, which we denote here by $\nu_{\textrm{eq};\mu}$. Recall that the equilibrium measure $\nu_{\textrm{eq};\mu}$ is defined to be the unique minimizer of the functional
\beq \no
I(\nu)=\iint \log\frac{1}{|x-y|} {\rm d}\nu(x){\rm d}\nu(y),
\eeq
minimized over all probability measures with support in $\supp(\mu)$.
Under some additional assumptions, more can be said about the convergence $\frac{K_n(x,x)}{n} \dmu(x) \rightarrow \textrm{d}\nu_{\textrm{eq};\mu}(x)$. If $\supp(\mu)$ contains an interval, $I$, then the restriction of $\nu_{\textrm{eq};\mu}$ to $I$ is absolutely continuous w.r.t.\ Lebesgue measure, and its Radon-Nikodym derivative, denoted by $\rho$, is real analytic there. In case that $\mu$ is absolutely continuous on $I$ and has continuous and positive weight there then the weak convergence \eqref{eq:regularconv} may be replaced by pointwise convergence in the following sense:
\begin{proposition}[Totik] \label{prop:totik}
Assume $\dmu(x)=w(x)dx+\dmu_{\textrm{sing}}$ is compactly supported and regular. Assume further that $w$ is continuous and positive on an interval, $I$, for which also $\mu_{\textrm{sing}}(I)=0$. Then
\beq \label{eq:totikconv}
\lim_{n \rightarrow \infty} \frac{K_n(x,x)}{n}w(x)=\rho(x)
\eeq
uniformly for $x \in I$, where $\rho(x)=\frac{d\nu_{\textrm{eq};\mu}(x)}{dx}$.
\end{proposition} 

Proposition \ref{prop:totik} was proved by Totik in \cite{totikCD}, extending many earlier results. Its relevance for our work is that if \eqref{eq:totikconv} holds, then a concentration inequality for a scaled linear statistic immediately implies a local law of large numbers. Thus, \eqref{eq:totikconv} may be thought of as convergence of the \emph{mean} empirical distribution down to the smallest possible scale.

In the case of OPE's arising as eigenvalue distributions of invariant ensembles from random matrix theory, $\mu$ depends on $n$ and (up to scaling) has the form $\dmu_n(x)=\frac{1}{Z_n}e^{-2nQ(x)}{\rm d}x= w_n(x) {\rm d}x$ where $Q$ is a `nice' real valued function and $Z_n$ is a normalization constant ($Q$ is often a polynomial). A result analogous to Proposition \ref{prop:totik} has been proved for varying $\mu$ with this form also by Totik in \cite{totik1}. 
\begin{proposition}[Totik] \label{prop:totik1}
Let $\Sigma \subseteq \bbR$ be a finite union of intervals and suppose that $w(x)=e^{-Q(x)}$ is a continuous function on $\Sigma$ satisfying $\lim_{|x| \rightarrow \infty }|x| e^{-Q(x)} =0$ if $\Sigma$ is not compact. Let $\nu_Q$ be the unique minimizer of the functional
\beq \no
I_Q (\nu)=\iint \log\frac{1}{|x-y|} {\rm d}\nu(x){\rm d}\nu(y)+2\int Q(x) {\rm d}\nu(x)
\eeq
$($minimized over all probability measures with support in $\Sigma)$ and assume that on a neighborhood of an interval, $J \subseteq \Sigma$, $\nu_Q$ is absolutely continuous and has a continuous density, $\rho_Q(x)$. 
Then
\beq \label{eq:varyingconv}
\lim_{n \rightarrow \infty} \frac{K_n\left(w^{2n};x,x\right)}{n} w^{2n}(x)=\rho_Q(x)
\eeq
uniformly on $J$, where $K_n(w^{2n}; \cdot, \cdot)$ is the $n$'th Christoffel-Darboux kernel for $\dmu_n(t)=w^{2n}(t){\rm d}t=e^{-2nQ(t)}{\rm d}t$. 
\end{proposition}

Weak convergence of $\frac{K_n\left(w^{2n};x,x \right)w^{2n}(x)}{n}{\rm d}x$ to ${\rm d}\nu_{Q}(x)$ has been shown earlier, e.g.\ in \cite{BdMPS}, to hold under fairly weak conditions on $Q$.

Regarding fluctuations on the microscopic scale, the famous universality conjecture states that for $x^*$ at which $\mu$ is sufficiently `nice' (namely, at least $\mu$ is absolutely continuous in a neighborhood of $x^*$ and $w>0$ there)
\beq \label{eq:universality}
\lim_{n\to \infty} \frac{\wti{K}_n\left(x^*+\frac{x}{\wti{K}_n\left(x^*,x^* \right)},x^*+\frac{y}{\wti{K}_n\left(x^*,x^* \right) }\right)}{\wti{K}_n(x,x)}=\frac{\sin \pi(x-y)}{\pi (x-y)}.
\eeq
where 
\beq \no
\wti{K}_n(x,y)=w(x)^{1/2} w(y)^{1/2}K_n(x,y).
\eeq
(In the case of varying $\mu$, $w$ above is replaced by $w_n$).

The name `universality' comes from random matrix theory, where the conjecture was originally made that the local correlations of eigenvalues of a large random matrix depend only on the symmetry of the matrix and not on the particular properties of the distribution. A significant amount of effort has been invested to verify this conjecture for different models. We do not attempt a review here. We only note that substantial progress was made by using the Riemann-Hilbert approach, employed in this context for the first time in \cite{DKMVZ1,DKMVZ2}. For a survey of the topic, including further developments, we refer to \cite{Kuijlaars} and references therein.

Recently, Lubinsky has provided two methods for proving universality under much weaker conditions on the measure $\mu$. Roughly, the requirements for proving universality using the first method sum up to local absolute continuity of $\mu$, with a continuous weight, together with global regularity \cite{lubinsky1}. Though initially formulated for measures supported on an interval, Lubinsky's results for compactly supported measures were extended in papers by Findley \cite{findley}, Simon \cite{Simon-ext}, and Totik \cite{totik}. The second method \cite{lubinskyUniv2} seems to be somewhat more general, and has indeed been used (with some modification) by Avila, Last and Simon to prove universality for almost every point in the essential support of the absolutely continuous part of spectral measures for ergodic Jacobi matrices \cite{als}. Using ideas from spectral theory, it has recently been shown \cite{Bjat} that there are even purely singular continuous measures for which an appropriate version of universality holds uniformly on an interval. We summarize what we need in the following proposition.

\begin{proposition} \label{prop:compactuniversal}
Let $\dmu(x)=w(x){\rm d}x+\dmu_{\textrm{sing}}$ be a regular measure with compact support, $E \subseteq \bbR$. Suppose $I$ is a closed interval in the interior of $E$ such that $\mu$ is absolutely continuous on a neighborhood of $I$ and $w$ is continuous and nonvanishing on $I$. Then, uniformly for $x \in I$ and $a,b \in$ compact subsets of $\bbR$
\beq \label{eq:lubconv1}
\frac{\wti{K}_n \left(x+\frac{a}{\wti{K}_n(x,x)},x+\frac{b}{\wti{K}_n(x,x)}\right)}{\wti{K}_n(x,x)} \rightarrow \frac{\sin \pi\left(b-a \right)}{\pi \left(b-a \right)}
\eeq
as $n \rightarrow \infty$. 
\end{proposition}

This proposition is essentially both in \cite{Simon-ext} and \cite{totik}. It is not phrased in precisely this fashion in either of the papers, but it follows immediately from their results using Proposition \ref{prop:totik}.

As for universality for varying measures, we quote the following result of Levin and Lubinsky \cite{LeLu}

\begin{proposition}[Levin-Lubinsky] \label{prop:varyinguniversal}
Let $\Sigma \subseteq \bbR$ be a finite union of intervals and suppose that $w(x)=e^{-Q(x)}$ is a continuous function on $\Sigma$ satisfying $\lim_{|x| \rightarrow \infty }|x| e^{-Q(x)} =0$ if $\Sigma$ is not compact. Let $\nu_Q$ be the unique minimizer of the functional
\beq \no
I_Q (\nu)=\iint \log\frac{1}{|x-y|} {\rm d}\nu(x){\rm d}\nu(y)+2\int Q(x) {\rm d}\nu(x)
\eeq
$($minimized over all probability measures with support in $\Sigma)$. Let $\dmu_n=w^{2n}(x){\rm d}x$ and let $\wti{K}_n(x,y)=w^{n}(x)w^n(y)K_n(x,y)$ be the corresponding Christoffel-Darboux kernel. 

Assume that on a neighborhood of an interval, $J$, lying in the interior of $\supp(\nu_Q)$, $\nu_Q$ is absolutely continuous and has a continuous and positive density, $\rho_Q(x)$. Assume also that $Q'$ is continuous on a neighborhood of $J$ as well. Then, uniformly for $x \in J$, $a,b \in$ compact subsets of $\bbR$
\beq \label{eq:lubconv2}
\frac{\wti{K}_n \left(x+\frac{a}{\wti{K}_n(x,x)},x+\frac{b}{\wti{K}_n(x,x)}\right)}{\wti{K}_n(x,x)} \rightarrow \frac{\sin \pi\left(b-a \right)}{\pi \left(b-a \right)}
\eeq
as $n \rightarrow \infty$. 
\end{proposition}

%%%%%%%%%%%%%%%%%%%%%%%%%%%% SECTION 3 %%%%%%%%%%%%%%%%

\section{A general concentration inequality} \label{concentration}

Our first result  is a general concentration inequality for linear statistics for determinantal point processes  that have a self-adjoint projection operator as  a kernel. For completeness, we state the result in a somewhat general context. We will apply it later to OPE's.

Let $\Lambda$ be a locally compact Polish space and $\lambda$ a positive Radon measure on $\Lambda$. Let $K:\Lambda \times \lambda \to \bbC$ be a measurable function such that $K(y,x)=\overline{K(x,y)}$ and 
$$\int K(x,y) K(y,z) {\rm d}\lambda (y)=K(x,z),$$
for $x,z\in \Lambda$. In other words, the associated integral operator is a self-adjoint projection on $\mathbb{L}_2(\lambda)$. Then it is known \cite{HKPV,Sosh} that $K$ defines a determinantal point process on $\Lambda$ with reference measure $\lambda$. Let us assume for simplicity that $K$ is also of finite rank $N<\infty$ (which we have in the case of OPE). Then for any bounded function $\phi$ we have 
\begin{multline}\label{eq:determinantalgeneral}
\EE \left[ \prod (1+\phi(x_j)\right] \\=\sum_{k=0}^N \frac{1}{k!} \int \dots \int \phi(x_1)\cdots \phi(x_k) \det \left(K(x_i,x_j)\right)_{i,j=1}^k {\rm d}\lambda(x_1) \cdots {\rm d} \lambda (x_k).
\end{multline}
(In case $N=\infty$ this identity still holds but conditions on $\phi$ are needed to make sure that it makes sense. For example, if $K$ is locally square integrable, then \eqref{eq:determinantalgeneral} holds for bounded $\phi$ with compact support.)

The following theorem tells us that the linear statistics for determinantal point processes  with self-adjoint projection kernels have subexponential tails. A particular consequence of the theorem is  that the fluctuations can be bounded in terms of the variance only, which will be our primary focus in the following sections.

\begin{theorem}\label{th:generalexponentialbounds}
 Let $\Lambda$ be a locally compact Polish space and $\lambda$ a positive Radon measure on $\Lambda$. Let $K:\Lambda \times \lambda \to \bbC$ be a measurable function such that the associated integral operator is a self-adjoint projection  of finite rank  (i.e. $K^2=K$ and $K^*=K$). Then for any  bounded $f$   we have that the linear statistic $X_f$ satisfies
\beq \label{eq:generalexponentialbounds}
P(\left|X_f-\EE X_f \right|\geq \eps) \leq \begin{cases}
 2\exp \left(-  \frac{\eps^2 }{4 A \Var X_f}\right), &\text{if } \eps <\frac{2 A \Var X_f}{3 \|f\|_\infty} \\
2\exp \left(-  \frac{\eps }{6 \|f\|_\infty} \right), &\text{if } \eps \geq \frac{2A \Var X_f}{3 \|f\|_\infty }
\end{cases}\eeq
where  $A>0$ is a constant that  does not depend on $f,K$ or $\eps$. 
\end{theorem} 

\begin{remark}
By the discussion in Section 2.1, Theorem \ref{th:generalexponentialbounds} applies in particular to OPE's.
\end{remark}

The assumption that $K$ is of finite rank is made for simplicity and is not used in the proof in an essential way. In a more general formulation, one needs an assumption on $f$ to make sense of \eqref{eq:determinantalgeneral} with $\phi= e^f-1$. While the proof we present here is for self-adjoint projection kernels, it can be adjusted to hold in a more general context. For example, in \cite{DJ} a similar statement is proved  for a  case where the kernel is only approximately a projection operator and no longer self-adjoint. 

Theorem \ref{th:generalexponentialbounds} follows from the following lemma on Fredholm determinants that we believe is interesting in its own right. While we formulate the statement in a more general setting, it might help to keep in mind our original context, in which the Hilbert space is $\mathbb L_2(\lambda)$, $K$ is a finite rank self-adjoint projection, and $h$ is the operator of multiplication by the function $f$.

\begin{lemma}\label{lem:boundonmoment}
Let $K$ be a self-adjoint projection operator (i.e.\ $K^2=K$ and $K^*=K$) on a separable Hilbert space. Let $h$ be a bounded operator such that $hK$ and $Kh$ are of trace class.  Then there exists a constant $A>0$ such that for any bounded function $h$ we have
\beq \label{eq:boundonmoment}
\left|\log \det(1+({\rm e}^{ t h}-1)K)- t\Tr h K\right|\leq \frac12 A |t|^2 \left\|[h,K]\right\|_2^2,
\eeq
for $|t|\leq \frac{1}{3 \|h\|_\infty}$. Here $\|\cdot\|_2$ denotes the Hilbert-Schmidt norm, $\|\cdot\|_\infty$ denotes the operator norm, and $[h,K]=hK-Kh$ stands for the commutator of $h$ and $K$. 
\end{lemma}

\begin{proof}
By the assumption on $t$ we have  a bound on the operator norm $\|({\rm e}^{t h}-1) K\|_\infty<1$. Hence we can rewrite the determinant as a sum of traces 
\beq \no
\begin{split}
\det(1+({\rm e}^{t h}-1) K)&=\exp \Tr \log (1+({\rm e}^{t h}-1) K)\\
&=\exp \sum_{j=1}^\infty \frac{(-1)^{j+1}}{j} \Tr \left( ({\rm e}^{t h}-1) K\right)^j.
\end{split}
\eeq
By expanding the exponential in a Taylor series we obtain
\begin{align}\no
\det(1+({\rm e}^{t h}-1) K)=\exp \sum_{j=1}^\infty \frac{(-1)^{j+1}}{j}  \sum_{l_1,\ldots, l_j =1}^{\infty} t^{l_1+\cdots +l_j}\frac{\Tr  h ^{l_1} K \cdots h^{l_j} K}{l_1!\cdots l_j!}, 
\end{align}
which, by an extra reorganization, can be turned into
\begin{align}\label{eq:standard:)}
\log \det(1+({\rm e}^{t h}-1) K)=\sum_{m=1}^\infty t^m\sum_{j=1}^m \frac{(-1)^{j+1}}{j}  \sum_{\overset{l_1+\cdots +l_j=m}{l_i\geq 1}} \frac{\Tr  h^{l_1} K \cdots  h^{l_j} K}{l_1!\cdots l_j!}.
\end{align}
This identity is our starting point. 

First note that
\[\sum_{j=1}^m \frac{(-1)^{j+1}}{j}  \sum_{\overset{l_1+\cdots +l_j=m}{l_i\geq 1}} \frac{1}{l_1!\cdots l_j!}=0 , \qquad m\geq 2,\]
which follows from $x=\log (1+({\rm e}^{x}-1))$ by expanding the logarithm and exponential. 
Thus,
\beq \no
t \Tr hK=\sum_{m=1}^\infty t^m\sum_{j=1}^m \frac{(-1)^{j+1}}{j}  \sum_{\overset{l_1+\cdots +l_j=m}{l_i\geq 1}} \frac{\Tr  h^{m} K}{l_1!\cdots l_j!}.
\eeq
It follows that 
\begin{multline}\label{eq:boundonlinstatstep2}
\log \det(1+({\rm e}^{t h}-1) K)-t \Tr hK\\
=\sum_{m=1}^\infty t^m\sum_{j=1}^m \frac{(-1)^{j+1}}{j}  \sum_{\overset{l_1+\cdots +l_j=m}{l_i\geq 1}} \frac{\Tr  h^{l_1} K \cdots  h^{l_j} K -\Tr  h^m K}{l_1!\cdots l_j!} \\
=\sum_{m=2}^\infty t^m\sum_{j=2}^m \frac{(-1)^{j+1}}{j}  \sum_{\overset{l_1+\cdots +l_j=m}{l_i\geq 1}} \frac{\Tr  h^{l_1} K \cdots  h^{l_j} K -\Tr  h^m K}{l_1!\cdots l_j!},
\end{multline}
since the $j=1$ term always vanishes.

We claim that
\begin{align}\label{eq:claimdifftraces}
\left|\Tr h^{l_1} K \cdots  h^{l_j} K-\Tr  h^m K \right| \leq j m^2 \|  h\|_\infty^{m-2}   \|[  h,K]\|_2^2,
\end{align}
for any $l_1,\ldots,l_j$ and $m\geq 2 $ such that $l_1+\cdots+l_j=m$ and $l_i\geq 1$. Note that $hK$ and $Kh$ by assumption are of trace class and hence also Hilbert-Schmidt operators so that the right-hand side is finite. 

To prove \eqref{eq:claimdifftraces}, first fix $j=2$. A straightforward computation using properties of the trace and the fact that $K^2=K$ shows that 
\begin{align}\no
\Tr  h^{l_1} K  h^{l_2}K =\Tr h^m K +\frac{1}{2} \Tr [  h^{l_1},K] [  h^{l_2},K].
\end{align}
Hence we have 
\begin{align}\label{eq:estimatedifftraces}
\left|\Tr  h^{l_1} K h^{l_2}K -\Tr h^m K \right| \leq \|  [h^{l_1},K]\|_2\| [  h^{l_2},K]\|_2.
\end{align}
By writing 
\beq \label{eq:estimateonHScom}
[h^{l},K]=\sum_{j=1}^{l} h^{l-j}[h,K]h^{j-1}
\eeq
we see that
\beq \no
 \|  [h^{l},K]\|_2^2\leq  l^2 \|  h\|_\infty^{2(l-1)}  \|  [h,K]\|_2^2,
\eeq
which, together with \eqref{eq:estimatedifftraces}, implies 
\begin{align}\label{eq:estimatedifftraces2}
\left|\Tr  h^{l_1} K  h^{l_2}K -\Tr  h^m K \right|\leq \frac{1}{2}  l_1 l_2 \|  h\|_\infty^{m-2} \|  [h,K]\|_2^2.
\end{align}
Since $l_1,l_2\leq m$, we see that \eqref{eq:claimdifftraces} follows (in fact, with a slightly better estimate) for $j=2$.

To  prove \eqref{eq:claimdifftraces} for $j\geq 3$ we use the following identity, which can be verified using solely $K^2=K$ (we need $j \geq 3$ here),
\beq\no
  h^{l_1} K \cdots  h^{l_j} K =  h^{l_1} K \cdots  h^{l_{j-1}+l_j} K+h^{l_1} K \cdots  h^{l_{j-2}}  K [  h^{l_{j-1}},K][  h^{l_j},K].
\eeq 
This implies
\begin{multline} \no
\left|\Tr  h^{l_1} K \cdots  h^{l_j} K -\Tr  h^{l_1} K \cdots  h^{l_{j-1}+l_j} K
\right|\\
\leq \|  h\|_\infty^{l_1+\cdots +l_{j-2}} \|K\|^{j-2}_\infty \|[  h^{l_{j-1}},K]\|_2\|[  h^{l_j},K]\|_2.
\end{multline}
By using \eqref{eq:estimateonHScom}, $\|K\|_\infty=1$ and the fact that $l_{j-1},l_j\leq m$ we obtain 
\beq \no
\left|\Tr  h^{l_1} K \cdots  h^{l_j} K -\Tr  h^{l_1} K \cdots  h^{l_{j-1}+l_j} K
\right|
\leq m^2\|  h \|_\infty^{m-2} \|[  h,K]\|_2^2.
\eeq 
By iterating this inequality we arrive at \eqref{eq:claimdifftraces}.
\bigskip

Plugging \eqref{eq:claimdifftraces} into \eqref{eq:boundonlinstatstep2} we can estimate
\begin{multline}\no
\left| \log \det(1+({\rm e}^{t h}-1) K)- t \Tr   h K\right| \\
\leq \| [  h,K]\|_2^2  \sum_{m=2}^\infty |t|^m \|  h\|_\infty^{m-2}  m^2 \sum_{j=2}^m \sum_{\overset{l_1+\cdots +l_j=m}{l_i\geq 1}} \frac{1}{l_1!\cdots l_j!}.
\end{multline}
Now,
\begin{align}
 \sum_{j=2}^m \sum_{\overset{l_1+\cdots +l_j=m}{l_i\geq 1}} \frac{1}{l_1!\cdots l_j!}<\frac{m^m}{m!}\leq \frac{{\rm e}^m}{\sqrt{2 \pi m}},
\end{align}
(which follows by expanding $m^m=(1+\cdots+1)^m$ and using a standard estimate on $m!$). Hence 
\begin{multline}\no
\left| \log\det(1+({\rm e}^{t h}-1) K)-t \Tr   h K\right| 
\leq \| [  h,K]\|_2^2  \sum_{m=2}^\infty |t|^m \|  h\|_\infty^{m-2}  {\rm e}^m m^{3/2}\\
\leq |t|^2 \| [  h,K]\|_2^2 {\rm e}^2\sum_{m=0}^\infty ({\rm e}/3)^m (m+2)^{3/2}. 
\end{multline}

We have thus shown that there exists a constant, $A$ (independent of $h$ and  $K$), such that
\begin{align} \label{eq:boundonlinstatfinalstep} 
\left| \log \det(1+({\rm e}^{t h}-1) K)- t \Tr   h K\right|
\leq  \frac12 A |t|^2 \| [  h,K]\|_2^2  
\end{align}
 for $|t|\leq (3 \|h\|_\infty)^{-1}$. Moreover, we see that the constant $A$ can be taken to be $2{\rm e}^2\sum_{m=0}^\infty ({\rm e}/3)^m (m+2)^{3/2}$ (see also Remark \ref{rem:constant}). This finishes the proof.
\end{proof}

We are now ready to prove Theorem \ref{th:generalexponentialbounds}.

\begin{proof}[Proof of Theorem \ref{th:generalexponentialbounds}]
We use the exponential version of the Chebyshev's inequality for (real) random variables $Y$
\begin{align}\no
P(Y\geq \eps)\leq {\rm e}^{-t \eps} \mathbb E[{\rm e}^{t Y}],
\end{align}
for $\eps,t>0$, which implies
\begin{align} \label{eq:markovtwosided}
P(|Y|\geq \eps)\leq {\rm e}^{-t \eps}\left( \mathbb E[{\rm e}^{t Y}]+\EE[{\rm e}^{-t Y}]\right),
\end{align}
for $\eps,t>0$.  

We want to apply \eqref{eq:boundonmoment} with the following setup: We take $\mathbb L_2(\lambda)$ for the Hilbert space, $K$ is the self-adjoint  projection kernel which is of finite rank and hence also trace class, and $h$ is the multiplication operator with multiplier $f$.

First note that under these assumptions, for the linear statistic $X_f$,
\begin{align}\no
\EE X_f &=\Tr f K,\\
 \Var X_f &= \frac{1}{2} \| [f,K]\|_2^2,
\end{align}
which, like \eqref{eq:trace} and \eqref{eq:variance2}, follow from \eqref{eq:fredholm2}. In addition, by taking $\phi={\rm e}^{tf}-1$ in  \eqref{eq:fredholm2} we obtain
$$\EE \left[\exp t X_f\right]= \det(1+({\rm e}^{t f}-1) K_n,$$ so that
$$\EE \left[\exp t(X_f-\EE X_f) \right]=\left( \det(1+({\rm e}^{t f}-1) K_n\right) \exp(-t \Tr fK).$$

Applying \eqref{eq:markovtwosided} with $Y=X_f-\EE X_f$ and using \eqref{eq:boundonmoment}  we obtain 
\beq \no
\mathbb P(|X_{f}-
\EE X_f|\geq \eps) \leq 2 {\rm e}^{-\eps t + A t^2 \Var X_f}.
\eeq 
for $|t| \leq 1/(3 \|f\|_{\infty})$. By choosing $$t=\min \left(\frac{\eps}{2 A \Var X_f},\frac{1}{3 \|f\|_\infty}\right),$$
we obtain the statement.
\end{proof}

To apply Theorem \ref{th:generalexponentialbounds} in particular cases it is important  to control the variance. For processes with kernels that are self-adjoint projections there is a very general and crude bound for the variance, which leads to the following result (which is Theorem \ref{thm:LLN} in a more general context). 

\begin{theorem}\label{th:generalexponentialboundswithvariance}
 Let $\Lambda$ be a locally compact Polish space and $\lambda$ a positive Radon measure on $\Lambda$. Let $K:\Lambda \times \lambda \to \bbC$ be a measurable function such that the associated integral operator is a self-adjoint projection  of finite rank. Then for any  bounded $f$   we have that the linear statistic $X_f$ satisfies
 \beq \label{eq:generalvariance}
 \Var X_f\leq 2 r(K) \|f\|_\infty^2,
 \eeq
 where $r(K)$ is the rank of $K$. In particular, 
\beq \label{eq:generalexponentialboundswithvariance}
P(\left|X_f-\EE X_f \right|\geq \eps) \leq \begin{cases}
 2\exp \left(-  \frac{\eps^2 }{8 A r(K) \|f\|_\infty^2}\right), &\text{if } \eps <\frac{4 A  r(K) \|f\|_\infty}{3} \\
2\exp \left(-  \frac{\eps }{6 \|f\|_\infty} \right), &\text{if } \eps \geq \frac{4 A r(K) \|f\|_\infty}{3 }
\end{cases}\eeq
where  $A>0$ is a constant that  does not depend on $f,K$ or $\eps$. 
\end{theorem} 
\begin{proof}
Note that because $K$ has finite rank $r(K)$ and $K$ is a projection (so all eigenvalues are $0$ or $1$) we have
\beq
\iint K(x,y)^2 {\rm d}\lambda(y){\rm d}\lambda (x)=\int K(x,x) {\rm d}\lambda(x)=\Tr K=r(K).
\eeq 
By combining this with \eqref{eq:variance2} (which also holds for general determinantal point processes with a kernel that is a self-adjoint projection)  the conclusion for the variance \eqref{eq:generalvariance}  immediately follows. By substituting this into \eqref{eq:generalexponentialbounds} we obtain \eqref{eq:generalexponentialboundswithvariance}. 
\end{proof}

\begin{remark}
 The rank of $K$ equals the number of points in the process \cite[Th. 4]{Sosh}.
\end{remark}

%%%%%%%%%%%%%%%%%%%%%%%%%%%% SECTION 4 %%%%%%%%%%%%%%%%

\section{Simple bounds on the Variance} \label{sec:variance}

Theorem  \ref{th:generalexponentialbounds} reduces the problem of obtaining a concentration inequality for OPE to the problem of finding bounds on the variance. This section is devoted to obtaining such bounds for general measures $\mu$.
The first bound is the general bound \eqref{eq:generalvariance}, which is as weak as it is general, but combined with Theorem \ref{th:generalexponentialbounds}, it already implies Theorem \ref{thm:LLN}.

\begin{proposition} \label{prop:bounded}
For any bounded function, $f$, and any measure with finite moments, $\mu$,
\beq \label{eq:stupidboundvariance}
\Var X_f^{(n)} \leq 2 n \|f\|_\infty^2.
\eeq 
Moreover, the same bound holds for $X_{f,\alpha,x^*}^{(n)}$ for any $\alpha$ and any $x^*$, i.e.\
\beq \label{eq:stupidboundvariance1}
\Var X_{f,\alpha,x^*}^{(n)} \leq 2 n \|f\|_\infty^2.
\eeq
\end{proposition}

We can slightly improve this bound if we assume that $f$ is continuous.

\begin{proposition} \label{prop:contin}
Assume that $\mu$ has compact support and that $f$ is a bounded continuous function. Then
\beq \label{eq:variancecontinuous}
\Var X_f^{(n)} = o(n), \quad \text{ as } n\to \infty.
\eeq

The same conclusion holds for $\mu$ with non-compact support, if we restrict attention to uniformly continuous bounded $f$, and if we assume, in addition, that $\frac{1}{n}\left(\frac{\gamma_{n-1}}{\gamma_n}\right)^2 \rightarrow 0$ as $n \rightarrow \infty$.

Similarly, for a sequence of measures $\left\{\mu_n\right\}_{n=1}^\infty$, the same conclusion holds if we restrict attention to uniformly continuous bounded $f$, and if we assume, in addition, that $\frac{1}{n}\left(\frac{\gamma_{n-1}^{(n)}}{\gamma_n^{(n)}}\right)^2 \rightarrow 0$ as $n \rightarrow \infty$.
\end{proposition}

\begin{proof}
We start by treating the case of compactly supported $\mu$ and general bounded and continuous $f$.
Let $\varepsilon>0$ and let $1>\delta>0$ be such that if $x \in \supp(\mu)$ or $y \in \supp(\mu)$, and $|x-y|<\delta$ then $|f(x)-f(y)|<\varepsilon$ (since $[\min(\supp(\mu))-1,\max(\supp(\mu))+1]$ is compact, $f$ is uniformly continuous there). Now, write
\beq \label{variance3}
\begin{split}
& \Var\left(  X_{f}^{(n)}\right) \\ 
&=\frac{1}{2}\int \int \left(f(x)-f(y) \right)^2 K_n(x,y)^2 \dmu(y)\dmu(x) \\
&=\frac{1}{2}\int \dmu(x) \int_{|x-y|<\delta} \left(f(x)-f(y) \right)^2 K_n(x,y)^2 \dmu(y) \\
&\quad +\frac{1}{2}\int \dmu(x) \int_{|x-y| \geq \delta} \left(f(x)-f(y) \right)^2 K_n(x,y)^2 \dmu(y) \\
& \leq \frac{1}{2}\varepsilon^2 \int \dmu(x) \int K_n(x,y)^2 \dmu(y) \\
&\quad + 2\|f\|_\infty
\int \dmu(x) \int_{|x-y| \geq \delta} K_n(x,y)^2 \dmu(y).
\end{split}
\eeq

By the reproducing property, 
\beq \no
\frac{1}{2}\varepsilon^2 \int \dmu(x) \int K_n(x,y)^2 \dmu(y) = \frac{1}{2}\varepsilon^2 \int \dmu(x) K_n(x,x)
=\frac{n}{2}\varepsilon^2
\eeq
As for the second term, we use the Christoffel-Darboux formula \eqref{eq:CDForm} to write
\beq \no
\begin{split}
&\int \dmu(x) \int_{|x-y| \geq \delta} K_n(x,y)^2 \dmu(y) \\
&\quad \leq  
\int \dmu(x)\int \frac{(x-y)^2}{\delta^2} K_n(x,y)^2 \dmu(y) \\
& \quad \leq \frac{1}{\delta^2}\frac{\gamma_{n-1}^2}{\gamma_n^2}
\int \dmu(x) \int \left(p_n(x)p_{n-1}(y)-p_n(y)p_{n-1}(x) \right)^2 \dmu(y) \\
&\quad\leq \frac{1}{\delta^2}\frac{\gamma_{n-1}^2}{\gamma_n^2}
\int \dmu(x)\left( p_n(x)^2+p_{n-1}(x)^2\right) \\
&\quad=\frac{2}{\delta^2 }\frac{\gamma_{n-1}^2}{\gamma_n^2}.
\end{split}
\eeq 

Thus
\beq \no
\Var\left(  X_{f}^{(n)}\right) \leq \frac{n}{2}\varepsilon^2+\frac{4 \|f\|_\infty}{\delta^2 }\frac{\gamma_{n-1}^2}{\gamma_n^2}.
\eeq

Since $\mu$ has compact support, $\frac{1}{n}\left(\frac{\gamma_{n-1}}{\gamma_n}\right)^2 \rightarrow 0$. Thus, $\limsup_{n \rightarrow \infty}n^{-1} \Var\left( X_{f}^{(n)}\right) <\varepsilon^2$ for any $\varepsilon$, which means
$\Var\left( X_{f}^{(n)}\right)=o(n)$.

The proofs for $\mu$ with non-compact support and for a sequence of measures follow along the same lines.

\end{proof}

If $\mu$ has compact support and $f$ is uniformly continuous and bounded, this strategy can be stretched to give the same result for $X_{f,\alpha,x^*}$ for any $\alpha<1/2$. More generally:

\begin{proposition} \label{prop:LocalContin}
Assume $f$ is uniformly continuous and bounded, then for any $x^* \in \bbR$ and any $\alpha$ such that 
$\frac{1}{n^{1-2\alpha}}\left( \frac{ \gamma_{n-1}^{(n)}}{\gamma_n^{(n)}}\right)^2\rightarrow 0$ as $n \rightarrow \infty$, we have
\beq \label{eq:variancecontinuous1}
\Var X_{f,\alpha, x^*}^{(n)} = o(n), \quad \text{ as } n\to \infty.
\eeq
\end{proposition}

\begin{proof}
As in the previous proof, fixing $\varepsilon>0$, we choose $\delta>0$ so that for any $x,y$ with $n^\alpha|x-y|<\delta$,
$\left| f\left(n^\alpha(x-x^*) \right)-f\left(n^\alpha(y-x^*) \right) \right| <\varepsilon$, and write

\beq \label{variance4}
\begin{split}
& \Var\left(  X_{f,\alpha,x^*}^{(n)}\right) \\ 
&=\frac{1}{2}\int \int \left(f\left(n^\alpha(x-x^*\right)-f\left(n^\alpha(y-x^*)\right) \right)^2 K_n(x,y)^2 \dmu(y)\dmu(x) \\
&=\frac{1}{2}\int \int_{|x-y|<\frac{\delta}{n^\alpha}}+\frac{1}{2}\int \int_{|x-y| \geq \frac{\delta}{n^\alpha}}  \\
&\leq  \frac{1}{2}\varepsilon^2 \int \dmu(x) K_n(x,x)+ 2\|f\|_\infty
\int \dmu(x)\int \frac{n^{2\alpha}(x-y)^2}{\delta^2} K_n(x,y)^2 \dmu(y) \\
&\leq \frac{n}{2}\varepsilon^2+2\frac{ n^{2\alpha}\|f\|_\infty}{\delta^2}\frac{\gamma_{n-1}^2}{\gamma_n^2}
\int \dmu(x)\left( p_n(x)^2+p_{n-1}(x)^2\right) \\
&=\frac{n}{2}\varepsilon^2+\frac{4n^{2\alpha} \|f\|_\infty}{\delta^2 }\frac{\gamma_{n-1}^2}{\gamma_n^2}.
\end{split}
\eeq
This implies that $\limsup_{n \rightarrow \infty}n^{-1} \Var\left( X_{f,\alpha,x^*}^{(n)}\right) <\varepsilon^2$ for any $\varepsilon$, which means $\Var\left( X_{f_\alpha,x^*}^{(n)}\right)=o(n)$.
\end{proof}

\begin{remark}
The proofs above can be modified to allow a finite number of points of discontinuity of $f$ (with $f$ remaining uniformly continuous away from these points), if we add the assumption that $\mu$ is continuous and that $K_n(x,x) \leq C n$ on neighborhoods of these points (in the macroscopic case) or $K_n(x,x) \leq C n$ on a neighborhood of $x^*$ (in the mesoscopic case). We omit this proof since we shall prove a similar result (albeit under more restrictive conditions) for all $0<\alpha<1$ later.
\end{remark}

Note that due to the Vandermonde determinant in the definition \eqref{eq:defOPE}, the random configuration for an OPE can be thought of as a Coulomb gas.  Hence we expect to see some repulsion between the points. When compared with the situation in which we take the points $\lambda_j$ to be i.i.d.\ random variable (instead of the OPE) it is reasonable to expect that this repulsion reflects itself in a smaller variance.  This partly explains why we are able to improve \eqref{eq:stupidboundvariance} and \eqref{eq:stupidboundvariance1} to \eqref{eq:variancecontinuous} and \eqref{eq:variancecontinuous1} for continuous functions. The repulsion becomes even more apparent when we restrict attention to Lipschitz functions.

\begin{proposition} \label{prop:lipschitz} 
Let $\mu$ be any measure, $x^*\in \bbR$, $\alpha\geq 0$ and assume that $f$ satisfies
\beq \no
|f(x)-f(y)|\leq |f|_{\mathcal L} |x-y|
\eeq for some $|f|_{\mathcal L} \geq 0$. Then
\beq \label{eq:variancelip}
\Var X_{f,\alpha,x^*} ^{(n)}\leq|f|_{\mathcal L}^2 \left(\frac{\gamma_{n-1}}{\gamma_n} \right)^2 n^{2\alpha}.
\eeq
In particular, if $\mu$ has compact support, then $\Var X_f^{(n)}$ is bounded.
\end{proposition}
\begin{proof}
Inserting the Lipschitz condition into \eqref{eq:AlphaVariance} we obtain 
\beq
\begin{split}
\Var X_{f,\alpha,x^*}^{(n)} & \leq \frac{|f|_{\mathcal L}^2 n^{2\alpha}}{2} \iint (x-y)^2 K_n(x,y)^2 {\rm d}\mu(x) {\rm d}\mu(y) \\ 
&=|f|_{\mathcal L}^2 \left(\frac{\gamma_{n-1}}{\gamma_n} \right)^2 n^{2\alpha}.
\end{split}
\eeq
In the case that $\mu$ has compact support, $\frac{\gamma_{n-1}}{\gamma_n}$ is uniformly bounded in $n$. This proves the statement.
\end{proof}
\begin{remark}
We find it remarkable that the variance for Lipschitz linear statistics is bounded  in $n$ for any $\mu$ with compact support. No assumptions on local continuity or global regularity of $\mu$ are needed.
\end{remark}
\begin{remark}
We have stated \eqref{eq:variancelip} in the form that we have in order to show its applicability to general measures (or sequences of measures). For $\alpha=0$ it shows that one can bound the variance of Lipschitz functions by $\left(\frac{\gamma_{n-1}}{\gamma_n} \right)^2$. 
Note that for $0\leq \alpha < 1/3$,  the simple estimate \eqref{eq:variancelip} is better than the (not so simple) esimate \eqref{varianceConclusion} below, which assumes much more about the measure (but less about $f$). 
\end{remark}

Combining these results with Theorem \ref{th:generalexponentialbounds} we get the following, (in which {\bf I} is Theorem \ref{thm:LLN}).

\begin{theorem} \label{thm:generalmacroLLN} \  \\ 
{\bf I.}  There exists a universal constant, $A>0$, such that for any measure with finite moments, $\mu$, any bounded function, $f$, any $\varepsilon>0$, and any $n\in \bbN$, 
\beq \label{eq:expboundmacroAA}
P\left( \left|\frac{1}{n} X_f^{(n)}-\frac{1}{n}\EE X_f^{(n)} \right|\geq \eps \right)\leq 2 \exp \left(- n \eps \min\left( \frac{ \eps}{8A \|f\|_\infty^2}, \frac{1}{6 \|f \|_\infty} \right) \right).
\eeq
In particular, if $\mu_n$ is a sequence of such measures for which also $\frac{K_n(x,x)}{n}{\rm d} \mu_n(x)$ has a weak limit, $\nu$, and $f$ is bounded and continuous, then
 \beq \label{eq:LLNAA}
\lim_{n\to \infty} \frac{1}{n} X_f^{(n)}= \int f(x) \ {\rm d} \nu(x),
\eeq 
almost surely. \\
\noindent
{\bf II.} 
Assume that either $\mu$ has compact support and $f$ is continuous; or that the sequence $\{\mu_n\}_{n=1}^\infty$ satisfies $\frac{1}{n}\left(\frac{\gamma_{n-1}^{(n)}}{\gamma_n^{(n)}}\right)^2 \rightarrow 0$ as $n \rightarrow \infty$, and $f$ is uniformly continuous and bounded. Then for any $\varepsilon>0$
\beq \label{eq:expboundmacrocontAA}
P\left( \left|\frac{1}{n} X_f^{(n)}-\frac{1}{n}\EE X_f^{(n)} \right|\geq \eps \right)\leq 2 \exp \left(- \frac{n\eps }{6 \|f \|_\infty}  \right)
\eeq
for  $n \in \bbN$ sufficiently large.

\noindent
{\bf III.} Suppose that $\mu$ has compact support or $\mu=\mu_n$ with $\left(\frac{\gamma_{n-1}^{(n)}}{\gamma_n^{(n)}}\right)^2=\mathcal O (1)$ as $n \rightarrow \infty$.  There there exists a constant $A$ such  that for any Lipschitz function, $f$ with Lipschitz constant $|f|_{\mathcal L}$, and any $\eps>0$
\beq \label{eq:expboundmacrocont}
P\left( \left| X_f^{(n)}- \EE X_f^{(n)} \right|\geq \eps \right)\leq 2 \exp  \left(-  \eps \min\left( \frac{ \eps}{ 4A |f|_{\mathcal L}^2}, \frac{1}{6 \|f \|_\infty} \right) \right)
\eeq
for  $n \in \bbN$ sufficiently large.
\end{theorem}

\begin{proof} \ \\
{\bf I.}
Note that by replacing $\eps=\eps n$ in \eqref{eq:generalexponentialbounds} and taking the maximum of the two cases we have
\beq \label{eq:bigtheorema}
\begin{split}
&P\left(\frac{1}{n}|X_f^{(n)}-\EE \frac{1}{n} X_f^{(n)}|\geq \eps\right) \\
&\quad\leq 2 \max\left(\exp\left(-\frac{n^2\eps^2}{4 A \Var X_f}\right),\exp\left(- \frac{n \eps}{6\|f\|_\infty}\right)\right)
\\
&=2\exp\left(-\min  \left( \frac{n^2\eps^2}{4 A \Var X_f},\frac{n\eps}{6\|f\|_\infty}\right)\right).
\end{split}
\eeq
By \eqref{eq:stupidboundvariance} we can further estimate the first term in the minimum and obtain \eqref{eq:expboundmacroAA}.

Now let us also assume $\frac{1}{n} K_n(x,x) {\rm d} \mu(x)$ has a weak limit denoted by $\nu$. This means that for any continuous function $f$ and $\eps>0$ 
$$\left|\frac{1}{n} \EE X_f^{(n)}- \int f(x) \ {\rm d} \nu(x)\right|\leq \eps/2$$
for $n$ sufficiently large. Hence by the triangle inequality we have that 
\beq \no
P\left(\left| \frac{1}{n} X_f^{(n)}- \int f(x) \ {\rm d} \nu(x)\right|\geq \eps \right)\leq P\left(\frac{1}{n} \left|X_f^{(n)}- \EE X_f^{(n)}\right|\geq \eps/2 \right),
\eeq
for $n$ sufficiently large. By  \eqref{eq:expboundmacroAA} the latter is  exponentially small as $n\to \infty$. Therefore, by applying the Borel-Cantelli lemma we obtain the almost sure convergence. 

\noindent
{\bf II.}  Under the stated conditions, Proposition \ref{prop:contin} implies that 
$$\min  \left( \frac{n^2\eps^2}{4 A \Var X_f},\frac{n\eps}{6\|f\|_\infty}\right)=\frac{n\eps}{6\|f\|_\infty},$$
for $n$ sufficiently large. By inserting the latter into \eqref{eq:bigtheorema} we obtain \eqref{eq:expboundmacrocontAA}. 

\noindent
{\bf III. } This follows directly from \eqref{eq:generalexponentialbounds} and \eqref{eq:variancelip}. 
 \end{proof}

Note that there is no normalization in  \eqref{eq:expboundmacrocont}. This is due to the strong repulsion between the points.

Theorem \ref{thm:generalmacroLLN} is on the global scale. As noted in the Introduction, wherever the bound depends only on $\|f \|_\infty$, there is an immediate corollary for the mesoscopic scale. Nevertheless, for good measure, we state the meaningful result for $x^*$ in the bulk in the following theorem, which also contains the mesoscopic bound for Lipschitz $f$. 

\begin{theorem}\label{thm:generalsemiLLN1} \ \\
{\bf I.} There exists a universal constant, $A>0$, such that for any measure with finite moments, $\mu$, any $x^*\in \bbR$, any $0<\alpha<1$,  any bounded function, $f$,  any $\varepsilon>0$, and any $n\in \bbN$,
\beq \label{eq:expboundmeso}
P\left(\frac{1}{n^{1-\alpha}} \left| X_{f,\alpha,x^*}^{(n)}-\EE X_{f,\alpha,x^*}^{(n)} \right|\geq \eps \right)\leq 2 \exp \left(- \varepsilon \min \left( \frac{ \eps n^{1-2\alpha}}{ 8A \|f\|_\infty^2 }, \frac{n^{1-\alpha}}{6\|f\|_\infty} \right) \right)
\eeq

\noindent
{\bf II.} Suppose that $\mu$ has compact support or $\mu=\mu_n$ with $\left(\frac{\gamma_{n-1}^{(n)}}{\gamma_n^{(n)}}\right)^2=\mathcal O (1)$ as $n \rightarrow \infty$.  Then  for any  any $x^*\in \bbR$, any $0<\alpha$,  any Lipschitz function, $f$, with Lipschitz constant $|f|_{\mathcal L}$ and any $\eps>0$ 
\beq \label{eq:expboundmacrocontalpha}
P\left( \frac{1}{n^{\alpha}}\left| X_{f,x^*,\alpha}^{(n)}- \EE X_{f,x^*,\alpha}^{(n)} \right|\geq \eps \right)\leq 2 \exp  \left(-\eps \min\left( \frac{ \eps}{ 2A |f|_{\mathcal L}^2}, \frac{n^\alpha}{6 \|f \|_\infty} \right) \right).
\eeq
In particular, for sufficiently large $n$ $($depending on $\varepsilon)$,
\beq \no
P\left( \frac{1}{n^{\alpha}}\left| X_{f,x^*,\alpha}^{(n)}- \EE X_{f,x^*,\alpha}^{(n)} \right|\geq \eps \right)\leq 2 \exp  \left(- \frac{ \eps^2}{ 2A |f|_{\mathcal L}^2} \right).
\eeq
\end{theorem} 

\begin{remark}
The inequality \eqref{eq:expboundmeso}, which is the local analog of \eqref{eq:expboundmacroAA}, holds for all $0<\alpha$ but is meaningful only for $\alpha<1/2$.  The analog of \eqref{eq:expboundmacrocontAA} (where $f$ is assumed to be continuous) has no additional information. The inequality \eqref{eq:expboundmacrocontalpha} is the local analog of \eqref{eq:expboundmacrocont}.
\end{remark}

\begin{proof}
The exponential bound \eqref{eq:expboundmeso} follows from substituting \eqref{eq:stupidboundvariance1} into \eqref{eq:generalexponentialbounds} with $\eps$ replaced by $\eps n^{1-\alpha}$.  
The bound in \eqref{eq:expboundmacrocontalpha} follows by replacing $\eps$ by $n^{\alpha}\eps$ in \eqref{eq:generalexponentialbounds} and using \eqref{eq:variancelip}. 
\end{proof}

\begin{remark} \label{rem:convergence}
Assume that $\mu_n$ are all absolutely continuous on a neighborhood, $I$, of $x^*$, with weight $w_n$ there. Then if $f$ has compact support
\beq \no
\begin{split}
\frac{1}{n^{1-\alpha}}\EE X_{f,\alpha,x^*}^{(n)} & =\frac{1}{n^{1-\alpha}}\int f\left(n^\alpha \left(x-x^* \right) \right)K_n(x,x) \dmu_n(x)\\
&=\frac{1}{n^{1-\alpha}}\int_I f\left(n^\alpha \left(x-x^* \right) \right)K_n(x,x) w(x) {\rm d}x\\
&=\int f(s) \frac{K_n\left(x^*+\frac{s}{n^\alpha},x^*+\frac{s}{n^\alpha} \right)}{n}w \left(x^*+\frac{s}{n^\alpha} \right) {\rm d}s
\end{split}
\eeq
for sufficiently large $n$. If $\frac{K_n(x,x)}{n}w_n(x)$ has a uniform limit, $\rho(x)$ on $I$, and $\rho$ is continuous at $x^*$, we get from this that
\beq \no
\lim_{n \rightarrow \infty}\frac{1}{n^{1-\alpha}}\EE X_{f,\alpha,x^*}^{(n)}=\rho(x^*) \int f(s) {\rm d}s.
\eeq
Thus, in this case, the bound \eqref{eq:expboundmeso} implies a local law of large numbers for compactly supported $f$ and $\alpha<1/2$. Propositions \ref{prop:totik} and \ref{prop:totik1} describe conditions under which the above holds. We show below, however, that essentially under these conditions we can in fact get a local law of large numbers for $\alpha<1$. 

It is worth noting that if $f$ has support in an interval $I$ with $\mu_n$ satisfying the above conditions on $I$, then \eqref{eq:expboundmacroAA} implies that a law of large numbers is satisfied for $f$ bounded and not necessarily continuous.  
\end{remark}

\section{The $\alpha$-Nevai condition and bounds on the variance in all mesoscopic scales} \label{alphanevai}

We recall that a measure, $\mu$, is said to satisfy the Nevai condition at $x$ if for any continuous, compactly supported function, $f$,
\beq \label{eq:nevaicondition}
\int( f(y)-f(x)) \frac{K_n(x,y)^2}{K_n(x,x)} {\rm d} \mu(y)\to 0,
\eeq
as $n\to \infty$.  In other words, if the Nevai condition holds at $x$ then $\frac{K_n(x,y)^2}{K_n(x,x)}\dmu(y)$ converges to a $\delta$-measure at $x$. The connection between the Nevai condition and the variance of $X_f$ becomes apparent when one writes the variance as
\beq \no
\begin{split}
\frac{\Var X_f^{(n)}}{n}&=\int \left(\int (f(y)-f(x)) \frac{K_n(x,y)^2}{K_n(x,x)} {\rm d} \mu(y) \right)   f(x)  \frac{K_n(x,x)}{n}{\rm d}\mu(x).
\end{split}
\eeq
Thus, we see that the convergence to zero in \eqref{eq:variancecontinuous} may be thought of as an averaged Nevai condition. On the other hand, if the Nevai condition is known to hold uniformly and $\frac{K_n(x,x)}{n}$ is bounded on the support of $f$, then we get convergence to zero. For the macroscopic scale, this is of course useless, since we already know \eqref{eq:variancecontinuous}. However, we will use precisely this intuition to obtain the local law of large numbers.

As noted in the Introduction, the Nevai condition has been introduced by Nevai in \cite{nevai} in order to study ratio asymptotics for Christoffel functions, and has been extensively studied, both for its own intrinsic interest and because of its connection to problems in approximation theory and spectral theory \cite{BLS, CMN, dellavecchia, LO, lubinskyNevai, LN, Nevai2, NTZ, Szwarc, Szwarc2, Zhang}. Conjecture 1.4 of \cite{BLS} states that for any compactly supported $\mu$, the Nevai condition holds for $\mu$-a.e.\ $x$. As an offshoot of our analysis, we shall show that universality at $x^*$, together with absolute continuity of $\mu$ there, implies the Nevai condition holds at $x^*$. This implies, in particular, that the condition holds for a.e.\ point w.r.t.\ the absolutely continuous part of spectral measures of ergodic Jacobi matrices. 

In order to control the variance of $X_{f,\alpha,x^*}^{(n)}$ we need a scaled version of the Nevai condition. It will also be useful in our discussion to let $\mu$ depend on $n$. 

\begin{definition} \label{def:alpha-nevai}
Let $\alpha>0$. We say that a sequence of measures $\{\mu_n\}_{n=1}^\infty$, satisfies the $\alpha$-Nevai condition at $x^*$ if for any $f$, continuous with compact support,
\beq \label{eq:uniform-alpha-nevai}
\lim_{n \rightarrow \infty} \int \left(f(s)-f\left(n^\alpha \left( y-x^* \right) \right) \right)\frac{K_n\left(x^*+ \frac{s}{n^\alpha},y \right)^2}{K_n\left(x^*+ \frac{s}{n^\alpha},x^*+\frac{s}{n^\alpha} \right)} \dmu_n(y) =0
\eeq
uniformly for $s$ in compact sets of $\bbR$, where $K_n(x,y) = K_n(x,y;\mu_n)$.
\end{definition}

\begin{remark}
If $\mu_n \equiv \mu$ is constant, we say that $\mu$ satisfies the $\alpha$-Nevai condition at $x^*$.
\end{remark}

Note that for $s=0$, $\alpha=0$, the $\alpha$-Nevai condition is just the Nevai condition (replace $f$ by $f(\cdot+x^*)$). Thus, an $\alpha$-Nevai condition implies convergence of $\frac{K_n\left(x^*,y \right)^2}{K_n(x^*,x^*)} d\mu_n(y)$ to a delta measure at $x^*$ at a fast rate.
\begin{proposition} \label{prop:alphaNevai1}
If $\{\mu_n\}$ satisfies an $\alpha$-Nevai condition at $x^*$ for some $\alpha>0$, then $\mu_n$ satisfies the Nevai condition at $x^*$.
\end{proposition}

\begin{proof}
Note that it is enough to show that 
\beq \label{eq:concmass}
\lim_{n \rightarrow \infty} \int_{y \in \left(x^*-\delta, x^*+\delta\right)}\frac{K_n \left(x^*,y \right)^2}{K_n \left(x^*,x^* \right)}\dmu_n(y) =1
\eeq 
for any $\delta>0$.

Fix $\delta>0$. By choosing $f$ so that $f(0)=1$, $0\leq f(t)<1$ for any $t \neq 0$, and $\supp(f) \subseteq (-\delta,\delta)$, it follows from \eqref{eq:uniform-alpha-nevai} with $s=0$ that
\beq \no
\lim_{n \rightarrow \infty} \int f\left(n^\alpha\left(y-x^* \right) \right)\frac{K_n \left(x^*,y \right)^2}{K_n \left(x^*,x^* \right)}\dmu_n(y) =1,
\eeq 
\beq \no
 \int_{y \notin \left(x^*-\delta/n^\alpha, x^*+\delta/n^\alpha\right)} f\left(n^\alpha\left(y-x^* \right) \right)\frac{K_n \left(x^*,y \right)^2}{K_n \left(x^*,x^* \right)}\dmu_n(y)=0,
\eeq
and
\beq \no
\int_{y \in \left(x^*-\delta/n^\alpha, x^*+\delta/n^\alpha \right)} f\left(n^\alpha\left(y-x^* \right) \right)\frac{K_n \left(x^*,y \right)^2}{K_n \left(x^*,x^* \right)}\dmu_n(y) \leq 1.
\eeq
This implies that 
\beq \no
\lim_{n \rightarrow \infty} \int_{y \in \left(x^*-\delta/n^\alpha, x^*+\delta/n^\alpha\right)}\frac{K_n \left(x^*,y \right)^2}{K_n \left(x^*,x^* \right)}\dmu_n(y) =1,
\eeq 
and in particular \eqref{eq:concmass} holds. We are done.
\end{proof}

For a measure $\dmu(x)=w(x){\rm d}x+\dmu_{{\rm sing}}(x)$ and $K_n(x,y)$ the corresponding Christoffel-Darboux kernel, recall that
\beq \no
\wti{K}_n(x,y)=w(x)^{1/2} w(y)^{1/2}K_n(x,y).
\eeq
In case $\mu_n$ depends on $n$, we let 
\beq \no
\wti{K}_n(x,y)=w_n(x)^{1/2} w_n(y)^{1/2}K_n(x,y)
\eeq
with the understanding that $K_n$ is the $n$'th Christoffel-Darboux kernel for $\mu_n$.
The reason that the $\alpha$-Nevai condition is useful, is the following

\begin{theorem} \label{VarianceDecay}
Let $0<\alpha<1$. Let $\{\mu_n\}_{n=1}^\infty$ be a sequence of probability measures on $\bbR$, and write $\dmu_n(x)=w_n(x) {\rm d}x+\dmu_{n,{\rm sing}}(x)$. Assume $x^* \in I \subseteq \supp(\mu)$ for some interval, $I$, so that: \\
(i) for all $n$, $\mu_{n,{\rm sing}}(I)=0$. \\
(ii) $\sup_{n \in \bbN}\sup_{t \in I} \frac{\wti{K}_n(t,t)}{n}=C<\infty$.\\
(iii) $\{\mu_n\}_{n=1}^\infty$ satisfies the $\alpha$-Nevai condition at $x^*$.\\
Then for any bounded, compactly supported function, $f$, with a finite number of points of discontinuity, 
\beq \label{varianceConclusion}
n^{\alpha-1}\Var \left( X_{f,\alpha,x^*}^{(n)}\right)=o \left(1\right)
\eeq
as $n \rightarrow \infty$.
\end{theorem}

\begin{proof}
For simplicity of notation, assume $x^*=0$ (this can be achieved by a shift of $f$) and let $X_{f,\alpha}^{(n)}=X_{f,\alpha,0}^{(n)}$. For $f$ continuous with $J=\supp(f)$ compact, write
\beq \label{variance1}
\begin{split}
& \Var\left( X_{f,\alpha}^{(n)}\right) \\ 
&=\int \int \left( f(n^\alpha x)^2- f(n^\alpha x) f(n^\alpha y)\right) K_n(x,y)^2 \dmu_n(x) \dmu_n(y)  \\
&=\int \int f(n^\alpha x) \left(f(n^\alpha x)-f(n^\alpha y) \right)K_n(x,y)^2\dmu_n(x) \dmu_n(y) \\
&=\int \int f(n^\alpha x) \chi_{J}(n^\alpha x)\left(f(n^\alpha x)-f(n^\alpha y) \right) K_n(x,y)^2\dmu_n(x) \dmu_n(y). \\
\end{split}
\eeq
For $n$ sufficiently large, $n^\alpha x \in J$ implies $x \in I$. Thus, for sufficiently large $n$, we may change variables $s=n^\alpha x$ to get
\beq \label{variance2}
\begin{split}
& n^{\alpha-1}\Var\left(X_{f,\alpha}^{(n)}\right) \\ 
&=n^{\alpha-1}\int_{x \in J/n^\alpha} \int f(n^\alpha x) \left(f(n^\alpha x)-f(n^\alpha y) \right) K_n(x,y)^2 \dmu_n(y)w_n(x){\rm d}x \\
&=\frac{1}{n}\int_{s \in (n^\alpha I)\cap J} \int f(s) \left(f(s)-f(n^\alpha y) \right) K_n\left(\frac{s}{n^\alpha},y\right)^2 \dmu_n(y)w_n\left(\frac{s}{n^\alpha}\right){\rm d}s \\
&=\frac{1}{n}\int_{s \in (n^\alpha I)\cap J} f(s)w_n\left(\frac{s}{n^\alpha} \right)\left( \int \left(f(s)-f(n^\alpha y) \right) K_n\left(\frac{s}{n^\alpha},y\right)^2 \dmu_n(y)\right){\rm d}s \\
&=\int_{s \in (n^\alpha I) \cap J} f(s)\frac{\wti{K}_n\left(\frac{s}{n^\alpha},\frac{s}{n^\alpha} \right)}{n}H(s){\rm d}s
\end{split}
\eeq
where
$H(s)= \int \left(f(s)-f(n^\alpha y) \right) \frac{K_n\left(\frac{s}{n^\alpha},y\right)^2}{K_n \left(\frac{s}{n^\alpha},\frac{s}{n^\alpha} \right)} \dmu_n(y)$. But the $\alpha$-Nevai condition at $x^*=0$ says that $H(s) \rightarrow 0$ as $n \rightarrow \infty$, uniformly for $s \in J$. Moreover, condition $(ii)$ of the theorem says that $\left|f(s)\frac{\wti{K}_n\left(\frac{s}{n^\alpha},\frac{s}{n^\alpha} \right)}{n}\right|$ is uniformly bounded for $s \in n^\alpha I$. Since the integration is carried out over the compact set $J$, we see that $\lim_{n \rightarrow \infty}n^{\alpha-1} \Var\left( X_{f,\alpha}^{(n)} \right)=0$.

Now, suppose that $f$ has a finite number of points of discontinuity. Let $t_1,\ldots t_k$ be these points and let $\varepsilon>0$. Then we can write
\beq \no
f=g+h
\eeq 
where $g$ is continuous with compact support and $h$ satisfies
\beq \no
\supp(h) \subseteq \bigcup_{i=1}^k \left[t_i-\frac{\varepsilon}{4 C \|f\|_\infty^2 k},t_i+\frac{\varepsilon}{4 C \|f\|_\infty^2 k} \right] \equiv \wti{J},
\eeq
$h$ is continuous except at the points $t_i$, and also $\| h \|_\infty \leq \|f \|_\infty$.
Clearly \eqref{varianceConclusion} holds for $g$. As for $h$, write
\beq \no
\begin{split}
& n^{\alpha-1}\Var\left(X_{h,\alpha}^{(n)}\right) \\ 
&=\frac{1}{n}\int_{s \in (n^\alpha I)\cap \wti{J}} h(s)w\left(\frac{s}{n^\alpha} \right)\left( \int \left(h(s)-h(n^\alpha y) \right) K_n\left(\frac{s}{n^\alpha},y\right)^2 \dmu_n(y)\right)ds \\
&\leq \frac{2 \|h \|_\infty^2}{n} \int_{s \in (n^\alpha I)\cap \wti{J}} w_n\left(\frac{s}{n^\alpha} \right)\left( \int  K_n\left(\frac{s}{n^\alpha},y\right)^2 \dmu_n(y)\right)ds \\
& \leq \frac{2 \|f \|_\infty^2}{n} \int_{s \in (n^\alpha I)\cap \wti{J}} w_n\left(\frac{s}{n^\alpha} \right)  K_n\left(\frac{s}{n^\alpha},\frac{s}{n^\alpha}\right)ds \\
&=2 \|f \|_\infty^2 \int_{s \in (n^\alpha I)\cap \wti{J}} \frac{\wti{K}_n\left(\frac{s}{n^\alpha},\frac{s}{n^\alpha}\right)}{n}ds \\
&\leq 2C \|f\|_\infty^2\int_{s \in \wti{J}}{\rm d}s \leq \varepsilon.
\end{split}
\eeq
Thus 
\beq \no
\begin{split}
& \limsup_{n \rightarrow \infty} n^{\alpha-1}\Var\left(X_{f,\alpha}^{(n)}\right) \\
&=\limsup_{n \rightarrow \infty} n^{\alpha-1}\Var\left(X_{g,\alpha}^{(n)}\right)+\limsup_{n \rightarrow \infty} n^{\alpha-1}\Var\left(X_{h,\alpha}^{(n)}\right) \leq \varepsilon
\end{split}
\eeq
for any $\varepsilon>0$, which implies $\lim_{n \rightarrow \infty}n^{\alpha-1} \Var\left( X_{f,\alpha}^{(n)} \right)=0$.
\end{proof}

We proceed to formulate a criterion for $\alpha$-Nevai.
 
\begin{theorem} \label{alphaNevai}
Let $\{\mu_n\}_{n=1}^\infty$ be a sequence of probability measures on $\bbR$, and write $\dmu_n(x)=w_n(x) {\rm d}x+\dmu_{n,{\rm sing}}(x)$. Let $x^* \in I \subseteq \supp(\mu)$ for some interval, $I$, so that: \\
(i) for all $n$, $\mu_{n,{\rm sing}}(I)=0$ and $w_n(x^*)>0$. \\
(ii) There exists a function, $\rho$, continuous on $I$ and satisfying $\rho(x^*)>0$, so that uniformly for $x \in I$, 
\beq \label{eq:totikconv1}
\frac{\wti{K}_n(x,x)}{n} \rightarrow \rho(x)
\eeq 
as $n \rightarrow \infty$. \\
(iii) Uniformly for $x \in I$ and $a,b$ in compact sets in $\bbR$
\beq \label{eq:lubconv}
\frac{\wti{K}_n \left(x+\frac{a}{\wti{K}_n(x,x)},x+\frac{b}{\wti{K}_n(x,x)}\right)}{\wti{K}_n(x,x)} \rightarrow \frac{\sin \pi\left(b-a \right)}{\pi \left(b-a \right)}
\eeq
as $n \rightarrow \infty$. \\
Then for any $0 \leq \alpha <1$, $\mu$ satisfies the $\alpha$-Nevai condition at $x^*$.
\end{theorem}

\begin{proof}
Again, for simplicity of notation we assume $x^*=0$. 
Note that since $\int \frac{K_n\left( \frac{s}{n^\alpha},y \right)^2}{K_n\left( \frac{s}{n^\alpha},\frac{s}{n^\alpha} \right)}\dmu(y)=1$, showing $\alpha$-Nevai at $0$ is equivalent to showing
\beq \label{alpha-nevai1}
\lim_{n \rightarrow \infty} \int f\left(n^\alpha y \right)\frac{K_n\left( \frac{s}{n^\alpha},y \right)^2}{K_n\left( \frac{s}{n^\alpha},\frac{s}{n^\alpha} \right)}\dmu(y) =f(s)
\eeq
uniformly for $s$ in compact sets of $\bbR$. Now, for the integrand to be nonzero $n^\alpha y$ has to be in $\supp(f)$. For sufficiently large $n$, this implies that $y \in I$. Thus, writing $n^\alpha y = r$, we get for such $n$
\beq \label{alpha-nevai2}
\begin{split}
& \int f\left(n^\alpha y \right)\frac{K_n\left( \frac{s}{n^\alpha},y \right)^2}{K_n\left( \frac{s}{n^\alpha},\frac{s}{n^\alpha} \right)}\dmu(y)  = \int f\left(n^\alpha y \right)\frac{K_n\left( \frac{s}{n^\alpha},y \right)^2}{K_n\left( \frac{s}{n^\alpha},\frac{s}{n^\alpha} \right)} w(y){\rm d}y\\
&= \int f\left(r\right)\frac{K_n\left( \frac{s}{n^\alpha},\frac{r}{n^\alpha} \right)^2}{n^\alpha K_n\left( \frac{s}{n^\alpha},\frac{s}{n^\alpha} \right)}w\left(\frac{r}{n^\alpha} \right){\rm d}r.
\end{split}
\eeq

We shall first show that for any $\delta, \varepsilon>0$ there exists $N$ such that for any $n \geq N$, for all $s \in$ a fixed compact set $\subseteq \bbR$.
\beq \label{alpha-nevai3}
\int_{s-\delta}^{s+\delta}\frac{K_n\left( \frac{s}{n^\alpha},\frac{r}{n^\alpha} \right)^2}{n^\alpha K_n\left( \frac{s}{n^\alpha},\frac{s}{n^\alpha} \right)}w\left(\frac{r}{n^\alpha} \right)dr>1-\varepsilon.
\eeq
Thus, fix $\delta, \varepsilon>0$ and let $M>0$ be so large that
\beq \label{alpha-nevai4}
\int_{-M}^M \frac{\sin (\pi t)^2}{(\pi t)^2}dt > 1-\frac{\varepsilon}{2}.
\eeq

From the uniform convergence in \eqref{eq:totikconv1} and \eqref{eq:lubconv} (both in $x \in I$ and for $a,b$ in compacts), the continuity of $\rho$ at $0$ and the equicontinuity (for $t \in$ a compact set of $\bbR$) of the family $\frac{\sin(\pi t (\cdot))}{\pi t (\cdot)}$, it follows that for any sequence $x_n \rightarrow 0$,
\beq \label{eq:uniformuniv1}
\frac{\wti{K}_n \left(x_n+\frac{a}{n},x_n+\frac{b}{n}\right)}{\wti{K}_n(x_n,x_n)} \rightarrow \frac{\sin \pi \rho(0)\left(b-a \right)}{\pi \rho(0) \left(b-a \right)}
\eeq
uniformly for $a,b$ in compact subsets of $\bbR$.

Now, with the change of variables $r=s+\frac{t}{n^{1-\alpha}}$ we get
\beq \label{alpha-nevai6}
\begin{split}
& \int_{s-\delta}^{s+\delta}\frac{K_n\left( \frac{s}{n^\alpha},\frac{r}{n^\alpha} \right)^2}{n^\alpha K_n\left( \frac{s}{n^\alpha},\frac{s}{n^\alpha} \right)}w\left(\frac{r}{n^\alpha} \right){\rm d}r \\
&=\int_{-n^{1-\alpha}\delta}^{n^{1-\alpha}\delta}\frac{K_n\left( \frac{s}{n^\alpha},\frac{s}{n^\alpha}+\frac{t}{n} \right)^2}{n K_n\left( \frac{s}{n^\alpha},\frac{s}{n^\alpha} \right)}w\left(\frac{s}{n^\alpha}+\frac{t}{n} \right)dt \\
&=\int_{-n^{1-\alpha}\delta}^{n^{1-\alpha}\delta}\frac{\wti{K}_n\left( \frac{s}{n^\alpha},\frac{s}{n^\alpha}+\frac{t}{n} \right)^2}{ \wti{K}_n\left( \frac{s}{n^\alpha},\frac{s}{n^\alpha} \right)^2} \frac{\wti{K}_n\left( \frac{s}{n^\alpha},\frac{s}{n^\alpha} \right)}{n}{\rm d}t \\
& \geq \int_{-M/\rho(0)}^{M/\rho(0)}\frac{\wti{K}_n\left( \frac{s}{n^\alpha},\frac{s}{n^\alpha}+\frac{t}{n} \right)^2}{ \wti{K}_n\left( \frac{s}{n^\alpha},\frac{s}{n^\alpha} \right)^2} \frac{\wti{K}_n\left( \frac{s}{n^\alpha},\frac{s}{n^\alpha} \right)}{n}{\rm d}t 
\end{split}
\eeq
for sufficiently large $n$. But \eqref{eq:uniformuniv1} and \eqref{eq:totikconv1} imply that 
\beq \no
\begin{split}
&\lim_{n \rightarrow \infty}  \int_{-M/\rho(0)}^{M/\rho(0)}\frac{\wti{K}_n\left( \frac{s}{n^\alpha},\frac{s}{n^\alpha}+\frac{t}{n} \right)^2}{ \wti{K}_n\left( \frac{s}{n^\alpha},\frac{s}{n^\alpha} \right)^2} \frac{\wti{K}_n\left( \frac{s}{n^\alpha},\frac{s}{n^\alpha} \right)}{n}{\rm d}t  \\
& =\int_{-M/\rho(0)}^{M/\rho(0)}  \frac{\sin \pi \rho(0)\left( t\right)}{\pi \rho(0) \left(t\right)}\rho(0) {\rm d}t 
=\int_{-M}^{M}  \frac{\sin \pi \left( t\right)}{\pi \left(t\right)} {\rm d}t >1-\frac{\varepsilon}{2}
\end{split}
\eeq
by \eqref{alpha-nevai4}. This together with \eqref{alpha-nevai6} imply \eqref{alpha-nevai3}.

It remains to show that \eqref{alpha-nevai3} implies the theorem. For this purpose, let $g_n(r)=\frac{K_n\left( \frac{s}{n^\alpha},\frac{r}{n^\alpha} \right)^2}{n^\alpha K_n\left( \frac{s}{n^\alpha},\frac{s}{n^\alpha} \right)}w\left(\frac{r}{n^\alpha} \right)$. By \eqref{alpha-nevai1} and \eqref{alpha-nevai2}, we want to show that
\beq \label{alpha-nevai10}
\lim_{n \rightarrow \infty}\int f(r)g_n(r)dr=f(s)
\eeq
locally uniformly in $s \in \bbR$.

Let $\eta >0$. Let $\varepsilon<\min \left(\frac{\eta}{3 \|f\|_\infty},1 \right)$ and let $\delta>0$ be so small that for any $t \in  [s-\delta,s+\delta]$, $|f(s)-f(t)|<\frac{\eta}{3}$. Finally, let $f_{+,\delta}=\max\{f(t) \mid t \in [s-\delta,s+\delta]\}$ and $f_{-,\delta}=\min\{f(t) \mid t \in [s-\delta,s+\delta]\}$. 

Now, for $n$ such that \eqref{alpha-nevai3} holds, we have
\beq \label{alpha-nevai11}
\left|\int_{\bbR \setminus [s-\delta,s+\delta]}f(r)g_n(r){\rm d}r\right| < \varepsilon \|f\|_\infty<\frac{\eta}{3}
\eeq
since $\int g_n(r)dr \leq 1$. On the other hand, for such $n$,
\beq \no
f_{-,\delta}-\frac{\eta}{3}\leq f_{-,\delta}-\|f\|_\infty \varepsilon<f _{-,\delta}(1-\varepsilon) \leq \int_{s-\delta}^{s+\delta}f(r) g_n(r)dr \leq f_{+,\delta}
\eeq
which implies
\beq \label{alpha-nevai12}
f(s)-\frac{2\eta}{3}< \int_{s-\delta}^{s+\delta}f(r) g_n(r){\rm d}r <f(s)+\frac{\eta}{3}.
\eeq

Combining \eqref{alpha-nevai11} and \eqref{alpha-nevai12} we see that for sufficiently large $n$,
\beq \no
\begin{split}
&\left|\int f(r)g_n(r){\rm d}r-f(s) \right| \\
&\quad \leq \left|\int_{s-\delta}^{s+\delta}f(r) g_n(r){\rm d}r-f(s) \right|+\left|\int_{\bbR \setminus [s-\delta,s+\delta]}f(r)g_n(r){\rm d}r\right| \\
&\quad < \frac{2\eta}{3}+\frac{\eta}{3}=\eta.
\end{split}
\eeq
We are done.
\end{proof}
\begin{remark}
Clearly, one may take $\mu_n \equiv \mu$, constant in the theorems above, to get the appropriate theorems for a fixed measure.
\end{remark}

Before discussing the local law of large numbers, we digress to discuss the implications of our method for the Nevai condition. Theorem \ref{alphaNevai} together with Proposition \ref{prop:alphaNevai1} imply that any measure satisfying conditions \emph{(i)--(iii)} at $x^*$ also satisfies the Nevai condition there. But in fact, less is needed for just the Nevai condition:

\begin{theorem} \label{thm:weaknevai}
Let $\dmu(x)=w(x) {\rm d}x+\dmu_{{\rm sing}}(x)$. Let $x^* \in \supp(\mu)$ so that: \\
(i) The point $x^*$ is a Lebesgue point of $\mu$ in the sense that $w(x^*)>0$, 
\beq \label{eq:strongLeb1}
\lim_{\varepsilon \rightarrow 0^+}\frac{\mu_{{\rm sing}} \left(x^*-\varepsilon,x^*+ \varepsilon \right)}{2 \varepsilon}=0,
\eeq 
and
\beq \label{eq:strongLeb2}
\lim_{\varepsilon \rightarrow 0^+}\int_{x^*-\varepsilon}^{x^*+\varepsilon}\frac{ \left|w(x)-w(x^*)\right|}{2 \varepsilon}{\rm d}x=w(x^*).
\eeq  
(ii) Uniformly for $a,b$ in compact sets in $\bbR$
\beq \label{eq:lubconv222}
\frac{K_n \left(x^*+\frac{a}{\wti{K}_n(x^*,x^*)},x^*+\frac{b}{K_n(x^*,x^*)}\right)}{K_n(x^*,x^*)} \rightarrow \frac{\sin \pi\left(b-a \right)}{\pi \left(b-a \right)}
\eeq
as $n \rightarrow \infty$. \\
Then $\mu$ satisfies the Nevai condition at $x^*$.
\end{theorem}
\begin{proof}
It suffices to show that for any $\delta$ and any $\varepsilon$, for sufficiently large $n$,
\beq \label{eq:nevai}
\int_{x^*-\delta}^{x^*+\delta} \frac{K_n\left(x^*,y \right)^2}{K_n\left(x^*,x^*\right)} \dmu (y) \geq 1-\varepsilon.
\eeq
First, note that, since $\mu\{x^*\}=0$, $\wti{K}_n(x^*,x^*) \rightarrow \infty$ as $n \rightarrow \infty$. Thus, for any fixed $M>0$ and any $\delta>0$, for all sufficiently large $n$, $\frac{M}{\wti{K}_n(x^*,x^*)}<\delta$. Now, choose $M$ as in \eqref{alpha-nevai4} and write
\beq
\begin{split}
&\int_{x^*-\frac{M}{\wti{K}_n(x^*,x^*)}}^{x^*+\frac{M}{\wti{K}_n(x^*,x^*)}} \frac{K_n\left(x^*,y \right)^2}{K_n\left(x^*,x^*\right)} \dmu (y) \\
&=\int_{x^*-\frac{M}{\wti{K}_n(x^*,x^*)}}^{x^*+\frac{M}{\wti{K}_n(x^*,x^*)}} \frac{K_n\left(x^*,y \right)^2}{K_n\left(x^*,x^*\right)} \dmu_{{\rm sing}} (y)+
\int_{x^*-\frac{M}{\wti{K}_n(x^*,x^*)}}^{x^*+\frac{M}{\wti{K}_n(x^*,x^*)}} \frac{K_n\left(x^*,y \right)^2}{K_n\left(x^*,x^*\right)}  w(y) {\rm d}y.
\end{split}
\eeq

For the first term, note that, by \eqref{eq:lubconv222}, there exists $C>0$ so that $ \sup_{y \in \left [x^*-\frac{M}{\wti{K}_n(x^*,x^*)},x^*+\frac{M}{\wti{K}_n(x^*,x^*)} \right], n \in \bbN} \left| \frac{K_n \left(x^*,y \right)}{K_n\left(x^*,x^* \right)}\right| \leq C$. Thus
\beq \no
\begin{split}
&\int_{x^*-\frac{M}{\wti{K}_n(x^*,x^*)}}^{x^*+\frac{M}{\wti{K}_n(x^*,x^*)}} \frac{K_n\left(x^*,y \right)^2}{K_n\left(x^*,x^*\right)} \dmu_{{\rm sing}} (y) \\
&\quad=\int_{x^*-\frac{M}{\wti{K}_n(x^*,x^*)}}^{x^*+\frac{M}{\wti{K}_n(x^*,x^*)}} \frac{K_n\left(x^*,y \right)^2}{K_n\left(x^*,x^*\right)^2}K_n(x^*,x^*)  \dmu_{{\rm sing}} (y) \\
&\quad \leq C^2 (M+1) \frac{K_n(x^*,x^*)}{M+1} \mu_{{\rm sing}}\left(x^*-\frac{M}{\wti{K}_n(x^*,x^*)},x^*+\frac{M}{\wti{K}_n(x^*,x^*)} \right) \rightarrow 0
\end{split}
\eeq
as $n \rightarrow \infty$, by \eqref{eq:strongLeb1}.

For the second term change variables to $y=x^*+\frac{t}{\wti{K}_n(x^*,x^*)}$ to get
\beq \no
\begin{split}
&\int_{x^*-\frac{M}{\wti{K}_n(x^*,x^*)}}^{x^*+\frac{M}{\wti{K}_n(x^*,x^*)}} \frac{K_n\left(x^*,y \right)^2}{K_n\left(x^*,x^*\right)}  w(y) {\rm d}y \\
&\quad=\int_{-M}^{M} \frac{K_n\left(x^*,x^*+\frac{t}{\wti{K}_n(x^*,x^*)} \right)^2}{K_n\left(x^*,x^*\right)^2}  \frac{w\left(x^*+\frac{t}{\wti{K}_n(x^*,x^*)} \right)}{w(x^*)} {\rm d}t.
\end{split}
\eeq
Now, again by \eqref{eq:lubconv222}, for some $\wti{C}>0$,
\beq \no
\begin{split}
&\Bigg|\int_{-M}^{M} \frac{K_n\left(x^*,x^*+\frac{t}{\wti{K}_n(x^*,x^*)} \right)^2}{K_n\left(x^*,x^*\right)^2}  \frac{w\left(x^*+\frac{t}{\wti{K}_n(x^*,x^*)} \right)}{w(x^*)} {\rm d}t \\
&\quad-\int_{-M}^{M} \frac{K_n\left(x^*,x^*+\frac{t}{\wti{K}_n(x^*,x^*)} \right)^2}{K_n\left(x^*,x^*\right)^2} {\rm d}t \Bigg| \\
&\leq \wti{C} \int_{-M}^M\left|w\left(x^*+\frac{t}{\wti{K}_n(x^*,x^*)} \right)-w(x^*) \right| {\rm d}t \rightarrow 0
\end{split}
\eeq
by \eqref{eq:strongLeb2}. It follows that
\beq \no
\begin{split}
&\lim_{n \rightarrow \infty}\int_{x^*-\frac{M}{\wti{K}_n(x^*,x^*)}}^{x^*+\frac{M}{\wti{K}_n(x^*,x^*)}} \frac{K_n\left(x^*,y \right)^2}{K_n\left(x^*,x^*\right)}  w(y) {\rm d}y \\
&\quad=\int_{-M}^M\frac{\sin \pi t}{\pi t}{\rm d}t
\end{split}
\eeq
which, by the choice of $M$, finishes the proof.
\end{proof}

We close this digression into the topic of the Nevai condition with two immediate corollaries of Theorem \ref{thm:weaknevai}.

\begin{corollary}
Let $\mu$ be a compactly supported, finite measure on $\bbR$. Assume that $\mu$ is absolutely continuous on a neighborhood of $x^*$, $w$ is continuous at $x^*$ and $w(x^*)>0$. Assume further that
\beq \label{eq:wiggle}
\lim_{n \rightarrow \infty}\frac{K_n \left(x^*+\frac{a}{n},x^*+\frac{a}{n} \right)}{K_n(x^*,x^*)}=1
\eeq
uniformly for $a$ in compact subsets of $\bbR$. Then the Nevai condition holds at $x^*$.
\end{corollary}
\begin{proof}
The continuity of $w$ clearly implies condition $(i)$ of Theorem \ref{thm:weaknevai}. In addition, Lubinsky shows \cite[Theorem 1.1]{lubinskyUniv2} that these conditions imply $(ii)$ as well.
\end{proof}
\begin{remark}
Remark (d) of \cite{lubinskyUniv2} implies we that can replace \eqref{eq:wiggle} with the requirement that
\beq \no
\limsup_{n\rightarrow \infty} \frac{K_n\left(x^*+\frac{a}{n},x^*+\frac{a}{n} \right)}{K_n(x^*,x^*)} \leq 1.
\eeq
\end{remark}

\begin{corollary} 
If $J_\omega$ is an ergodic family of Jacobi matrices and $\mu_\omega$ is the corresponding family of spectral measures, then for a.e.\ $\omega$, the Nevai condition holds for a.e.\ $x^*$ with respect to the absolutely continuous part of $\mu_\omega$.
\end{corollary}
\begin{remark}
For background on ergodic Jacobi matrices and their spectral measures see \cite{als} and references therein.
\end{remark}
\begin{proof}
Avila, Last and Simon \cite{als} show that for a.e.\ $\omega$, condition $(ii)$ of Theorem \ref{thm:weaknevai} holds for a.e.\ $x^*$ with respect to the absolutely continuous part of $\mu_\omega$. In addition, condition $(i)$ holds for a.e.\ $x^*$ with respect to the absolutely continuous part of $\mu_\omega$ by standard arguments.
\end{proof}

From Theorems \ref{VarianceDecay} and \ref{th:generalexponentialbounds} we obtain

\begin{theorem}[Local Law of Large Numbers for General OPE]\label{LLLN1}
Let $\{\mu_n\}_{n=1}^\infty$ be a sequence of probability measures on $\bbR$, and write $\dmu_n(x)=w_n(x) {\rm d}x+\dmu_{n,{\rm sing}}(x)$. Let $x^* \in I \subseteq \supp(\mu)$ for some interval, $I$, so that: \\
(i) for all $n$, $\mu_{n,{\rm sing}}(I)=0$ and $w_n(x^*)>0$. \\
(ii) There exists a function, $\rho$, continuous on $I$ and satisfying $\rho(x^*)>0$, so that uniformly for $x \in I$, 
\beq \label{eq:totikconv122}
\frac{\wti{K}_n(x,x)}{n} \rightarrow \rho(x)
\eeq 
as $n \rightarrow \infty$. \\
(iii) Uniformly for $x \in I$ and $a,b$ in compact sets in $\bbR$
\beq \label{eq:lubconv333}
\frac{\wti{K}_n \left(x+\frac{a}{\wti{K}_n(x,x)},x+\frac{b}{\wti{K}_n(x,x)}\right)}{\wti{K}_n(x,x)} \rightarrow \frac{\sin \pi\left(b-a \right)}{\pi \left(b-a \right)}
\eeq
as $n \rightarrow \infty$. \\
Then for any compactly supported, bounded function, $f$, with a finite number of points of discontinuity, and any $\eps>0$ 
\beq \label{eq:LCI}
\mathbb P\left( n^{\alpha}  \left|\frac{X_{f,\alpha,x^*}^{(n)}}{n} 
-\frac{\EE f\left(n^\alpha (x-x^*) \right)}{n}\right|\geq \eps\right)\leq 2
{\rm e}^{-\frac{\eps  n^{(1-\alpha)}}{6 \|f\|_\infty} }, 
\eeq
for  $n\in \bbN$ sufficiently large. Hence in particular,
\beq \label{eq:asconvergence1}
\lim_{n \rightarrow \infty} n^\alpha \left (\frac{1}{n} X_{f,\alpha,x^*}^{(n)}-\int_I f\left(n^\alpha (x-x^*)\right)\rho(x){\rm d}x \right)=0,
\eeq
almost surely. 
\end{theorem}

\begin{proof}
A standard application of the Borel-Cantelli Lemma says that the local concentration inequality \eqref{eq:LCI} leads to
\beq \no
\lim_{n \rightarrow \infty} n^\alpha \left (\frac{1}{n} X_{f,\alpha,x^*}^{(n)}-\frac{\EE f\left(n^\alpha (x-x^*) \right)}{n}\right)=0,
\eeq
with probability one. Now recall that
\beq \no
\frac{\EE f\left(n^\alpha (x-x^*) \right)}{n^{1-\alpha}}=\int f \left(n^\alpha(x-x^*) \right) \frac{K_n(x,x)}{n^{1-\alpha}} \dmu(x) 
\eeq
and note that for any compactly supported function, $f$, for sufficiently large $n$
\beq \no
\int f \left(n^\alpha(x-x^*) \right) \frac{K_n(x,x)}{n^{1-\alpha}} \dmu(x) =\int_I f \left(n^\alpha(x-x^*) \right) \frac{K_n(x,x)}{n^{1-\alpha}} w(x) {\rm d}x 
\eeq
and so, if $f$ is also bounded,
\beq \label{eq:COVar}
\begin{split}
& \left|\int_I f \left(n^\alpha(x-x^*) \right) \frac{K_n(x,x)}{n^{1-\alpha}} w(x) {\rm d}x -n^\alpha \int_I f \left(n^\alpha(x-x^*) \right) \rho(x) {\rm d}x\right| \\
&=\left|\int_{s \in \supp(f)} f \left(s \right)\left( \frac{\wti{K}_n\left(x^*+\frac{s}{n^\alpha},x^*+\frac{s}{n^\alpha}\right)}{n}- \rho\left(x^*+\frac{s}{n^\alpha}\right)\right) {\rm d}s\right| \\
&\leq \|f \|_\infty \int_{s \in \supp(f)} \left| \frac{\wti{K}_n\left(x^*+\frac{s}{n^\alpha},x^*+\frac{s}{n^\alpha}\right)}{n}- \rho\left(x^*+\frac{s}{n^\alpha}\right)\right| {\rm d}s \rightarrow 0
\end{split}
\eeq
as $n \rightarrow \infty$, by the uniform convergence of $\frac{\wti{K}_n\left(x,x \right)}{n}$ to $\rho(x)$ and since $f$ has compact support. Thus \eqref{eq:LCI} implies \eqref{eq:asconvergence1}.

As for \eqref{eq:LCI}, this follows as in the proof of Theorem \ref{thm:generalmacroLLN} by replacing $\varepsilon$ in Theorem \ref{th:generalexponentialbounds} by $\varepsilon n^{1-\alpha}$ and using Theorem \ref{VarianceDecay} to see that for sufficiently large $n$
\beq \no
\min \left( \frac{n^{2-2\alpha}\eps^2}{4 A \Var X_f},\frac{n^{1-\alpha}\eps}{6\|f\|_\infty}\right)=\frac{n^{1-\alpha}\eps}{6\|f\|_\infty}.
\eeq
\end{proof}

 \begin{remark}
As the change of variables in \eqref{eq:COVar} shows, the scaling with $n^{1-\alpha}$ in \eqref{eq:LCI} is the correct one.  Indeed, $n^{\alpha} \int f\left(n^{\alpha} (x-x^*) \right) \rho(x) {\rm d}x$ converges to $\rho(x^*) \int f(s) {\rm d}s$ as $n\to \infty$.
\end{remark}

Using Propositions \ref{prop:totik} and \ref{prop:compactuniversal} we immediately get 

\begin{corollary}[=Theorem \ref{thm:LLLN}] \label{cor:compactLLLN}
Let $\dmu(x)=w(x){\rm d}x+\dmu_{{\rm sing}}(x)$ be a regular measure with compact support, $E \subseteq \bbR$. Suppose $I$ is a closed interval in the interior of $E$ such that $\mu$ is absolutely continuous on a neighborhood of $I$ and $w$ is continuous and nonvanishing on $I$. 
Then for any compactly supported, bounded function, $f$, with a finite number of points of discontinuity, and any $\eps>0$ 
\beq \label{eq:LCI3}
\mathbb P\left( n^{\alpha}  \left|\frac{X_{f,\alpha,x^*}^{(n)}}{n} 
-\frac{\EE f\left(n^\alpha (x-x^*) \right)}{n}\right|\geq \eps\right)\leq 2
{\rm e}^{-\frac{\eps  n^{(1-\alpha)}}{6 \|f\|_\infty} }, 
\eeq
for  $n\in \bbN$ sufficiently large. Hence in particular,
\beq \label{eq:asconvergence3}
\lim_{n \rightarrow \infty} n^\alpha \left (\frac{1}{n} X_{f,\alpha,x^*}^{(n)}-\int_I f\left(n^\alpha (x-x^*)\right){\rm d}\nu_{\textrm{eq},\mu} (x) \right)=0,
\eeq
almost surely. 
\end{corollary}

Using Propositions \ref{prop:totik1} and \ref{prop:varyinguniversal} we immediately get

\begin{corollary} \label{cor:varyingLLLN}
Let $\Sigma \subseteq \bbR$ be a finite union of intervals and suppose that $w(x)=e^{-Q(x)}$ is a continuous function on $\Sigma$ satisfying $\lim_{|x| \rightarrow \infty }|x| e^{-Q(x)} =0$ if $\Sigma$ is not compact. Let $\nu_Q$ be the unique minimizer of the functional
\beq \no
I_Q (\nu)=\iint \log\frac{1}{|x-y|} {\rm d}\nu(x){\rm d}\nu(y)+2\int Q(x) {\rm d}\nu(x)
\eeq
$($minimized over all probability measures with support in $\Sigma)$. Let $\dmu_n=w^{2n}(x){\rm d}x$ and let $\wti{K}_n(x,y)=w^{n}(x)w^n(y)K_n(x,y)$ be the corresponding Christoffel-Darboux kernel. 

Assume that on a neighborhood of an interval, $J$, lying in the interior of $\supp(\nu_Q)$, $\nu_Q$ is absolutely continuous and has a continuous and positive density, $\rho_Q(x)$. Assume also that $Q'$ is continuous on a neighborhood of $J$ as well. 

Then for any compactly supported, bounded function, $f$, with a finite number of points of discontinuity, and any $\eps>0$ 
\beq \label{eq:LCI3a}
\mathbb P\left( n^{\alpha}  \left|\frac{X_{f,\alpha,x^*}^{(n)}}{n} 
-\frac{\EE f\left(n^\alpha (x-x^*) \right)}{n}\right|\geq \eps\right)\leq 2
{\rm e}^{-\frac{\eps  n^{(1-\alpha)}}{6 \|f\|_\infty} }, 
\eeq
for  $n\in \bbN$ sufficiently large. Hence in particular,
\beq \label{eq:asconvergence3a}
\lim_{n \rightarrow \infty} n^\alpha \left (\frac{1}{n} X_{f,\alpha,x^*}^{(n)}-\int_I f\left(n^\alpha (x-x^*)\right){\rm d}\nu_{Q} (x) \right)=0,
\eeq
almost surely. 
\end{corollary}
\begin{remark}
It may seem that our treatment is restricted to absolutely continuous measures. However, that is not entirely true. If $\mu_n$ is a sequence of absolutely continuous measures that converges to a singular measure, $\mu$, such that $K_n(x,y;\mu_n) \equiv K_n(x,y;\mu)$, then Theorem \ref{LLLN1} shows that one might have a local concentration inequality and a local law of large numbers for the process defined by $K_n$ with respect to the measures $\mu_n$. Of course, for this to happen, the conditions of Theorem \ref{LLLN1} need to be met by $K_n$ and $\mu_n$. The family of measures described in \cite{Bjat} are an example where this happens. 

This does not yet mean that a local law of large numbers holds for the process defined by a singular $\mu$. We only get it for the sequence of processes defined by the sequence $\mu_n$ converging to $\mu$. This raises an interesting problem: If $\mu_1$ and $\mu_2$ have the same CD kernel, $K_n$, what is the relationship between the OPE defined by $\mu_1$ and $\mu_2$, at level $n$?
\end{remark}

\end{document}